\documentclass[10pt]{amsart}                          
\usepackage{epsfig}            
\usepackage{amssymb,amscd,bbm}     
\usepackage{amsthm,amsgen,amsmath,epic,psfrag}
\usepackage[dvips]{xypic}       

\xyoption{curve}                
\xyoption{rotate}               

\def\cup{\smallsmile}

\newcommand{\R}{{\mathbb R}}  
\newcommand{\Z}{{\mathbb Z}}  
\newcommand{\Q}{{\mathbb Q}}
       
\newcommand{\E}{{\mathcal E}}

\newcommand{\dinv}{d^{-1}}

\newcommand{\apl}{A_{PL}}

\newcommand{\cofreeE}{{\mathbb E}}

\newcommand{\s}{ }

\newcommand{\lie}{\mathcal{L}\!{\it ie}}    
\newcommand{\eil}{\mathcal{E}\!{\it il}}    

\newcommand{\inv}{ {{\text -}1}}      

\newcommand{\overtie}[2]{             
\begin{aligned} \displaystyle         %
\operatornamewithlimits{              %
 \begin{aligned} #1                   %
 \end{aligned} }_{#2}                 %
\end{aligned} }

\newcommand{\biggraphpp}[3]{ \ensuremath{
 \begin{xy}                           
  (0,-3)*+UR{\scriptstyle #1}="a",    
  (4,3)*+UR{\scriptstyle #2}="b",     
  (8,-3)*+UR{\scriptstyle #3}="c",    
  "a";"b"**\dir{-}?>*\dir{>},         %
  "b";"c"**\dir{-}?>*\dir{>}          %
 \end{xy}                             %
} }

\newcommand{\graphmp}[3]{ \ensuremath{
 \begin{xy}                           
  (0,-2)*+UR{\scriptstyle #1}="a",    
  (3,3)*+UR{\scriptstyle #2}="b",     
  (6,-2)*+UR{\scriptstyle #3}="c",    
  "b";"a"**\dir{-}?>*\dir{>},         
  "b";"c"**\dir{-}?>*\dir{>}          %
 \end{xy}                             %
} }

\newcommand{\linep}[2]{ \ensuremath{  %
 \begin{xy}                           
  (0,-2)*+UR{\scriptstyle #1}="a",    
  (6,2)*+UR{\scriptstyle #2}="b",     
  "a";"b"**\dir{-}?>*\dir{>},         
 \end{xy}                             %
} }

\theoremstyle{plain}                          
\newtheorem{theorem}{Theorem}[section]                          
\newtheorem{proposition}[theorem]{Proposition}                          
\newtheorem{lemma}[theorem]{Lemma}   
\newtheorem{corollary}[theorem]{Corollary}

\theoremstyle{definition}                          
\newtheorem{definition}[theorem]{Definition}  

\theoremstyle{remark}  
\newtheorem{example}[theorem]{Example}                                                          
\newtheorem{remark}[theorem]{Remark}

\newcommand{\refT}[1]{Theorem~\ref{T:#1}}
\newcommand{\refC}[1]{Corollary~\ref{C:#1}}
\newcommand{\refP}[1]{Proposition~\ref{P:#1}}

\newcommand{\refL}[1]{Lemma~\ref{L:#1}}
\newcommand{\refE}[1]{Equation~\ref{E:#1}}

\begin{document}

\title[Lie coalgebras II: Hopf invariants]{Lie coalgebras and  rational homotopy theory II: Hopf invariants}
\author[D. Sinha]{Dev Sinha}
\address{Department of Mathematics\\
University of Oregon\\
Eugene, OR
97403}
\email{dps@math.uoregon.edu}

\author[B. Walter]{Ben Walter} 
\address{
Department of Mathematics \\ Middle East Technical University, Northern Cyprus Campus \\
Kalkanli, Guzelyurt, KKTC, Mersin 10 Turkey
}
\email{benjamin@metu.edu.tr}

\subjclass{55P62; 16E40, 55P48.}
\keywords{Hopf invariants, Lie coalgebras, rational homotopy theory, graph cohomology}
\maketitle

We give a new solution of the ``homotopy periods'' problem, as highlighted by Sullivan \cite{Sull77}, 
which places explicit geometrically meaningful 
formulae  first dating back to Whitehead \cite{Whit47} in the context 
of Quillen's formalism for rational homotopy theory and Koszul-Moore duality
\cite{Quil70}.  


We build on \cite{SiWa06}, which uses 
graph coalgebras
to breathe combinatorial life into the category of  
differential graded
Lie coalgebras.  We use that framework to  construct a new
isomorphism of Lie coalgebras $\eta : H_{*-1}(\E(A^{*}(X))) \to {\rm Hom} 
(\pi_{*}(X), \Q)$ for $X$ simply connected.
Here $A^{*}(X)$ denotes a model for commutative rational-valued  
cochains on $X$, and  $\E$
is isomorphic to the Harrison complex.  While the existence of such an
isomorphism follows from Quillen's seminal work in rational  
homotopy theory,
giving a direct, explicit isomorphism has benefits in both theory and  
applications.   

On the calculational 
side, we are able to evaluate Hopf invariants on iterated Whitehead products 
in terms of the ``configuration pairing.''  
We can use this, for example, to take the well-known calculation of the 
rational homotopy groups of a wedge of  
spheres as a free
graded Lie algebra and give a geometric
algorithm to determine which  element of that algebra a given map
would correspond to.  
On the formal side, we are able to understand
the naturality of these maps in the long exact sequence of a  
fibration.  For applications, we can show for example that the rational
homotopy groups   of homogeneous spaces
are detected by classical linking numbers.
Ultimately  all of these Hopf invariants
are essentially generalized linking invariants, as we explain in Section~\ref{S:mfld}.

We proceed in two steps, first using the classical bar complex
to define integer-valued homotopy functionals which coincide with
evaluation of the cohomology of $\Omega  X$ on the looping of a map
from $S^{n}$ to $X$.   This was also the starting point of Hain's work  
\cite{Hain84} using Chen integrals,
but our definition of functionals is clearly distinct from his.    
We establish basic properties and give examples
using the classical bar complex. 
In the second part, we use the Harrison complex on commutative cochains,
and thus must switch to rational coefficients.
Using our graph coalgebraic presentation, 
we show that a product-coproduct formula established geometrically in
the bar complex descends to the duality predicted by Koszul-Moore theory.

Our basic, apparently new, observation is that calculations in  
bar complexes
yield the Hopf invariant
formula of Whitehead \cite{Whit47}, as well as those of Haefliger, Novikov and  
Sullivan.  This observation could have
been made fifty years ago.  Our approach
incorporates  a modern viewpoint by directly using  
Harrison-Andr\'e-Quillen homology,
the standard algebraic bridge from commutative algebras
to Lie coalgebras, with the new graphical presentation essential
for a self-contained development.   
One direction we plan to pursue further  
is the use of Hopf invariants
to realize Koszul-Moore duality isomorphisms in general.  A second  
direction we plan to pursue
is that of  spaces which are not simply connected, where  our graph 
coalgebra models seem relevant even for
$K(\pi, 1)$ spaces.

The problem of finding ``homotopy periods'' has been addressed before  by Boardman-Steer \cite{BoSt66}, Sullivan  
\cite{Sull77}, Haefliger
\cite{Haef78}, Hain \cite{Hain84} and Novikov \cite{Novi88}, using a  
wide range of tools.
We relate and
compare our approach with these at the end of the paper.
In summary, we view the cofree  Lie coalgebra functor as the best for
unifying formalism and geometry, showing how Koszul duality governs
homotopy groups through ``linking'' of cochain data as
explained in Section~\ref{S:mfld}.

\section{Hopf invariants from the bar complex.}

In these first sections when dealing exclusively with the bar complex, we only
need associative cochains and so work integrally. 
 Later when dealing with the Lie coalgebra model of a space,
we switch to exclusive use of commutative cochains and work over the rationals.
For consistency with historical practice we use $C^{*}(X)$ to denote the usual cochains
with cup product on a simplicial set $X$, or equivalently any subalgebra whose 
inclusion induces an isomorphism on cohomology.
Similarly, we let $A^{*}(X)$ denote the $PL$ forms on $X$ or equivalently a subalgebra
model.  It is an unfortunate accident of notational history that $A^{*}$ is commutative while
$C^{*}$ is associative but not commutative (except in an $E_{\infty}$ sense).

\begin{definition}
Let $B(R)$ denote the bar complex on an associative differential graded algebra $R$,
defined in the standard way as the total complex of a bicomplex
spanned by monomials $x_{1}| \cdots| x_{n}$, where the $x_{i}$ have positive degree,
multilinear in each variable.  The
``internal'' differential $d_{R}$ is
given by extending that of $R$ by the Leibniz rule, and the ``external'' differential $d_{\mu}$
is defined by removing bars and multiplying; we write $d_B$ for the total differential
in the bar complex when it is not otherwise clear by context.
If $x$ is a monomial in the bar complex, the internal degree of $x$ is the sum of the degrees
of its component elements, and the weight of $x$ is its number of component elements.  The total
degree of $x$ is its internal degree minus its weight.

Let $B(X)$ denote $B(C^{*}(X))$, and let $H^{*}_{B}(X)$ denote $H_{*}(B(C^{*}(X)))$.
\end{definition}

Throughout the paper, we generally suppress the suspension and desuspension
operators $s$ and $s^{\inv}$ which are used in the definition of the bar complex
and related complexes, 
as where they need to appear is always determined by context. We may include them when 
for example they facilitate computing signs.
Also, we will assume throughout that $X$ is a simply-connected space.  We leave it to 
a sequel to address the non-simply-connected case, which requires new foundational
understanding of Lie coalgebras.  Those foundations are being worked out by the second author.   
We have done some preliminary calculations, in which our techniques
work beyond the nilpotent setting.

\medskip

The classical work of Adams-Hilton and 
Eilenberg-Moore established that $H_{B}^{*}(X)$ is isomorphic
additively to the cohomology of the based loopspace of $X$. 
We now show 
that the homology of the bar complex is also the natural setting for Hopf invariants for
an arbitrary $X$, extending the  invariants for spheres and other suspensions.
Topologically, we are passing from a map $f: S^{n} \to X$ to its looping $\Omega f :
\Omega S^{n} \to \Omega X$, on which we evaluate cohomology classes from
the bar complex.  But the way in which we do the evaluation, and the properties
we derive, have not to our knowledge been previously considered.  We start
with a standard calculation in the bar complex for the sphere, including 
a proof because the central ingredient -- weight reduction -- yields 
 a method for explicit computation of Hopf invariants.

\begin{lemma}\label{L:key}
$H^{n-1}_{B}(S^n)$ is rank one, generated by an element of weight one corresponding
to the generator of $H^n(S^n)$. 
\end{lemma}

\begin{proof}
Suppose $\alpha \in B^{n-1}(S^{n})$ is a cycle.
 Since $\alpha$ has finitely many terms, its terms have maximal 
weight $k$.   Write $\alpha|_k$ for the weight $k$ terms of $\alpha$.  If $k>1$ then 
its internal differential $d_{C^{*}}(\alpha|_k) = 0$, so  
$\alpha|_k$ gives a cocycle in $\otimes_{k}\bar C^{*}(S^n)$.  By the K\"unneth theorem, 
$\otimes_{k}\bar C^{*}(S^n)$
has no homology in degree $n$, so $\alpha|_k$ is exact in $\otimes_{k}\bar C^{*}(S^n)$. 
Any choice of cobounding expression will determine a $\beta \in B(S^n)$ with 
$d_{C^{*}}\beta = \alpha|_k$.  Therefore $\alpha - d_B\beta$ is a lower weight expression 
in $B^{n-1}(S^n)$ homologous to $\alpha$.  
Inductively, we have that $\alpha$ is homologous to a cycle
of weight one, so the map from $H^{n}(S^{n})$ to $H^{n-1}_{B}(S^{n})$ including
the weight-one cocycles is surjective. 

Applying this weight-reduction argument to a cochain $\beta\in B^{n}(S^n)$ 
with $d \beta = x$
for $x$ weight one, we see the map from $H^{n}(S^{n})$ to $H^{n-1}_{B}(S^{n})$
is injective as well.
\end{proof}

\begin{definition}\label{D:int_B}
Let $\gamma \in B^{n-1}(S^n)$ be a cocycle.  Define $\tau(\gamma) \simeq \gamma$ 
to be a choice of weight one cocycle to which $\gamma$ is cohomologous.

Define $\int_{B(S^n)}$ to be the map from cocyles in $B^{n-1}(S^n)$ to $\mathbb{Z}$ given by
$\int_{B(S^n)} \gamma = \int_{S^n} \tau(\gamma)$, where $\int_{S^n}$ denotes evaluation
on the fundamental class of $S^n$.
\end{definition}

From \refL{key} it is immediate that the map $\int_{B(S^n)}$ is well defined and induces 
an isomorphism $H_{B}^{n-1}(S^{n}) \cong \Z$.

The standard way to use cohomology to define homotopy functionals is to pull back
and evaluate.  This is essentially how we define our generalized Hopf invariants, 
allowing for a homology
between the cocycle we pull back and one which we know how to evaluate.

\begin{definition}\label{D:hopfform}
Define the Hopf pairing $\langle \;,\;  \rangle_{\eta} :  H^{n-1}_{B}(X) \times  \pi_{n}(X) \to \Z$ 
by sending
$[\gamma] \times [f]$ to $\int_{B(S^{n})} f^{*}(\gamma)$.

We call $\tau(f^*(\gamma))$ the Hopf cochain (or form) of $\gamma$ pulled back by $f$.
We name the associated maps $\eta : H^{*-1}_{B}(X) \to {\rm Hom}(\pi_{*}(X), \Z)$ 
and $\eta^{\dagger} : \pi_{*}(X) \to {\rm Hom}(H^{*-1}_{B}(X), \Z)$.  We say
$\eta(\gamma) \in {\rm Hom}(\pi_{*}(X), \Z)$ is the Hopf invariant associated to $\gamma$. 
\end{definition}

A choice of Hopf cochain is not unique, but the corresponding Hopf invariant is well-defined.  It is
immediate that the Hopf invariants are functorial.   Moreover, the definitions
hold with any ring coefficients.  Topologically we have the following
interpretation.

\begin{proposition}\label{P:spacelevel}
The value of the Hopf invariant associated to a cocycle $\gamma$ in the bar complex on some map $f$, namely
$\eta(\gamma)(f)$, is equal to $[\gamma]\bigl(\Omega f_* [\Omega S^n]\bigr)$, the value of the cohomology class given
by $\gamma$ in $H^{n-1}(\Omega X)$ on the image under $\Omega f$ of the
fundamental class of $H_{n-1}(\Omega S^{n})$.
\end{proposition}

\subsection{Examples}\label{examples}

\begin{example}
A cocycle of weight one in $B(X)$ is just a closed cochain on $X$, which may be pulled back and 
immediately evaluated.  
Decomposable elements of weight one in $B(X)$ are null-homologous, 
consistent with the fact that products evaluate trivially on the Hurewicz homomorphism.
\end{example}

\begin{example}\label{e2}
Let $\omega$ be a generating $2$-cocycle on $S^{2}$ and $f : S^{3} \to S^{2}$.  
Then $\gamma =  -\s\omega | \s\omega$ is a cocycle in $B(S^{2})$ which $f$ pulls back
to $- \s f^{*} \omega | \s f^{*} \omega$, 
a weight two cocycle of total degree two on $S^{3}$.
Because $f^{*} \omega$ is closed and of degree two on $S^{3}$, it is exact. 
Let $\dinv f^{*} \omega$ be a choice of a cobounding cochain.  Then 
$$d_{B} \left(\s\dinv f^{*} \omega | \s f^{*} \omega \right) = 
\s f^{*} \omega |  \s f^{*} \omega \ +\ 
\s \bigl(\dinv f^{*} \omega \cup f^{*} \omega\bigr).$$
Thus $f^*\gamma$ is homologous to $\s\bigl(\dinv f^{*} \omega \cup f^{*} \omega\bigr)$,
and the corresponding Hopf invariant is $ \int_{S^{3}} {\dinv f^{*} \omega \cup f^{*} \omega}$,
which is  the classical formula for Hopf invariant given by Whitehead \cite{Whit47} 
(and generalized to maps from arbitrary domains by O'Neill \cite{ONei79}).

Expressions involving choices of $\dinv$ for some cochains will be a feature of
all of our formulae.  On the sphere one can make this 
explicit as in the proof of the Poincar\'e Lemma, as Sullivan pointed out  when defining similar formulae in
Section~11 of \cite{Sull77}, .
\end{example}

\begin{example}\label{EX:hopf on brackets}
Let $X= S^{n} \vee S^{m}$ and let $x$ be a cochain representative for a generator of  $H^{n}(S^{n})$ 
and similarly $y$ on $S^{m}$.
Then $\gamma = \s x|\s y$ is a cocycle in $B(X)$.  
Let $f : S^{n+m-1} \to S^{n} \vee S^{m}$ be the universal 
Whitehead product.  
More explicitly  let $p_{1} : D^{n} \times D^{m} \to S^{n}$ be projection 
onto $D^{n}$ followed by the canonical quotient map,
and let $p_{2} :  D^{n} \times D^{m} \to S^{m}$ be defined similarly.  Decompose 
$S^{n+ m - 1}$ as 
$$S^{n+m-1} = \partial (D^{n} \times D^{m}) = 
 D^{n} \times S^{m-1}\  \bigcup\  S^{n-1} \times D^{m}.$$
Then $f|_{D^{n} \times S^{m-1} } = p_{1}|_{D^{n} \times S^{m-1} }$ and 
$f|_{S^{n-1} \times D^{m}} = p_{2}|_{S^{n-1} \times D^{m}}$.

Proceeding as in the previous example, 
$\langle \gamma, f\rangle_\eta = (-1)^{|x|+1} \int_{S^{n+m-1}} \dinv f^{*}x \cup f^{*}y$.  
But these cochains extend to 
$D^{n} \times D^{m}$.  Namely,  $\dinv p_{1}^{*} x$ on $D^{n} \times D^{m}$ restricts 
to $\dinv f^{*} x$, and similarly $f^{*}y$ is the restriction of $p_{2}^{*} y$. 
We evaluate as follows:
$$(-1)^{|x|+1}\int_{S^{n+m-1}} \kern -20pt \dinv f^{*}x \cup f^{*}y =
\int_{\partial (D^{n} \times D^{m})}  \kern -25pt 
 (\dinv p_{1}^{*}x \cup p_{2}^{*}y)\big|_{\partial (D^{n} \times D^{m})} 
{=} \int_{D^{n} \times D^{m}} \kern -20pt p_{1}^{*} x \cup p_{2}^{*} y 
= \int_{D^{n}} \!\! x \cdot \int_{D^{m}}\!\!  y = 1.$$
The change in sign in the first equality above is due to the change in orientation 
on the fundamental
class induced by the isomorphism $S^{n+m-1} \cong \partial (D^n \times D^m)$.

 We conclude that the Hopf invariant of $\gamma$ detects the Whitehead product,
 which recovers a theorem from the third page of 
\cite{Haef78}.  This first case of evaluation of a Hopf invariant on a Whitehead product
will be generalized below.
\end{example}

\begin{example}\label{E:arbwt1}
For an arbitrary $X$ and cochains $x_{i}, y_{i}$ and $\theta$ on $X$ with 
$d x_{i}  = d y_{i} = 0$ and $d \theta = \sum (-1)^{|x_i|} x_{i} \cup y_{i}$, the cochain
$\gamma = \sum \s x_{i} | \s y_{i} + \s\theta \in B(X)$ is closed. 
The possible formulae for the Hopf invariant are all of the form  
$$\langle\gamma,\ f\rangle_\eta = \int_{S^{n}} \left( f^{*} \theta - 
   \sum \left((-1)^{|x_i|} t \cdot \dinv f^{*} x_{i} \cup f^{*} y_{i} 
         + (1-t) \cdot f^{*} x_{i} \cup \dinv f^{*} y_{i}\right) \right),$$
for some real number $t$.
This generalizes a formula given in the Computations section of \cite{GrMo81}
and is also present in \cite{Haef78}.

By choosing $t = \frac{1}{2}$ we see that reversing the order to consider $\sum \s y_{i} | \s x_{i}$
will yield the same Hopf invariant, up to sign.  Thus $\sum x_{i} | y_{i}  \mp y_{i} | x_{i}$ yields
a zero Hopf invariant.  Indeed, there are many Hopf invariants
which are zero, a defect which will be remedied by using the Lie coalgebraic bar construction.
\end{example}

\begin{example}\label{E:wt2}

In applications, Hopf forms are easily computed using a weight reduction technique
introduced in the proof of Lemma~\ref{L:key}.  
The bigrading of $B(X)$ is used as in the following example, illustrated in Figure~\ref{f1}.

Suppose there is a weight three cocycle in $B(X)$ of the form 
  $$\gamma = \s x_{1} | \s x_{2} | \s x_{3} - \s x_{12} | \s x_{3} + \s x_{123},$$
where $d x_{i} = 0$, $d x_{12} = x_{1} \cup x_{2}$, $d x_{123} = x_{12} \cup x_{3}$
and $x_{2} \cup x_{3} = 0$, $x_1$ has odd degree
and $x_2$ has even degree.  
Consider the element 
$$\alpha = \s\dinv f^*x_1 | \s f^* x_2 | \s f^* x_3 
   + \s\dinv\left(\dinv f^* x_1 \cup f^* x_2 - f^* x_{12} \right)| \s f^* x_3.$$ 
Here we observe that 
$(\dinv f^* x_1 \cup f^* x_2 - f^* x_{12} )$ is closed and thus exact in order to 
know that we may find a $\dinv$ for it in $C^{*}(S^n)$.
We express  $d_{B} (\alpha) = (d_{C^*X} + d_\mu)(\alpha)$ in  $B(S^n)$, 
which is naturally a second-quadrant bicomplex, as follows.

\begin{figure}[ht]\label{f1}

\begin{center} 
\psfrag{A}{$d_C$}
\psfrag{B}{$d_\mu$}
\psfrag{C}{$- \s f^*x_1 | \s f^* x_2 | \s f^* x_3$}
\psfrag{D}{$\s \dinv f^* x_1| \s f^* x_2 | \s f^* x_3$}
\psfrag{E}{$\s (\dinv f^* x_1 \cup f^*x_2) | \s f^* x_3$}
\psfrag{H}{$- \s(\dinv f^*x_1 \cup f^*x_2 -f^* x_{12}) | \s f^*x_3$}
\psfrag{F}{$\s\dinv\left(\dinv f^*x_1 \cup f^*x_2 - f^*x_{12} \right)| \s f^*x_3$}
\psfrag{G}{$- \s\dinv\left(\dinv f^*x_1 \cup f^*x_2 - f^*x_{12} \right) \cup f^* x_3$}

$$\includegraphics[width=15cm]{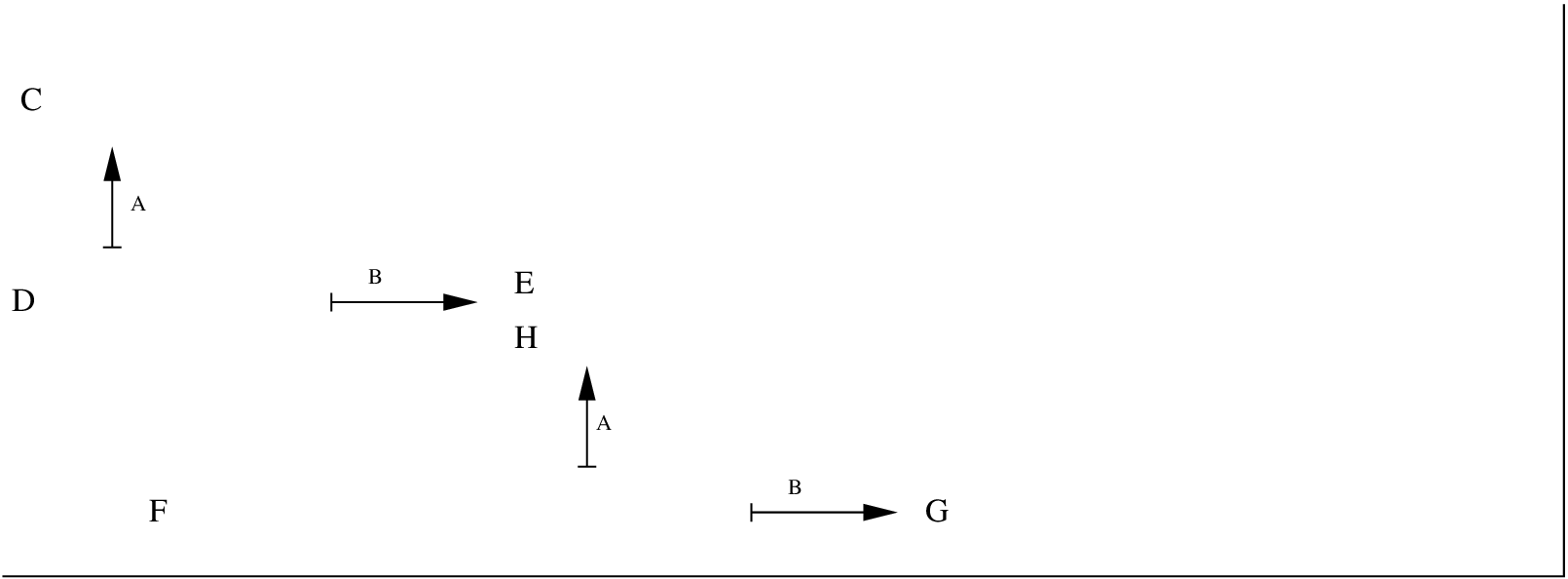}$$

\end{center}
 \caption{Calculation of $d_{B}(\alpha)$ in $B(S^{n})$ from Example~\ref{E:wt2}.}

\end{figure}

\medskip

Observe that 
$$f^*(\gamma) + d_{B} (\alpha) = 
 \s\bigl(f^* x_{123} - 
      \dinv\left(\dinv f^* x_1 \cup f^* x_2 - f^* x_{12} \right) \cup f^*x_3\bigr).$$ 
The right-hand side is a weight one cocycle in $B(S^n)$,
meaning that it gives a Hopf cochain of $\gamma$ pulled back by $f$.
\end{example}

\subsection{Hopf invariants as generalized linking numbers}\label{S:mfld}
We explain the geometric significance of these formulae, which is one of the main highlights of
this approach.
The classical Hopf invariant was understood as a linking number, and the present 
generalization is rightly understood in terms of linking numbers as well.   

The story is clearest when  $X$ is a manifold with a collection $\{W_{i}\}_{i=1}^{k}$ 
of proper oriented submanifolds of codimension two or greater, disjoint from each
other and the basepoint of $X$.    Denote associated
Thom cochains by $\omega_{i}$.   (There are various settings in which this can be done, the simplest
being to define a Thom cochain as  a 
``partially defined'' cochain whose value on a smooth singular chain is given by an intersection count.  We will
instead use de Rham theory, where constructions of Thom forms are well-documented as in 
Chapter 1.6 of \cite{BoTu82}).
Because the $W_{i}$
are disjoint, the $\omega_{i}$ can be chosen with disjoint support.    Then
$\gamma = \omega_{1} | \cdots | \omega_{k}$ is a cocycle in $B(X)$.  If
$f:S^d\to X$ where $d$ is the sum of the codimensions of
the $W_i$ minus $k$, then an associated Hopf form is
$$\tau(f^*(\gamma)) =
f^{*}\omega_{1} \wedge \dinv \left( f^{*} \omega_{2} \wedge \dinv \left( \cdots \cup \left(
f^{*}\omega_{k-1} \wedge \dinv f^{*}\omega_{k} \right) \cdots \right) \right).$$

Because $f$ is based we can   choose a diffeomorphism of 
$S^{d}\backslash\{\text{basepoint}\}$ with $\R^{d}$
and consider the  submanifolds $f^{-1}W_i \subset \R^{d}$.  In general if $\partial P = Q$ then we can construct $d^{\inv}$
of the Thom form of $Q$ as a Thom form of $P$.
In this case we define $d^{\inv} f^*\omega_k$ by first taking the Thom form of the manifold $E_{k}$ given by 
extending $f^{\inv} W_k$ upwards in $\R^d$ (say in the first coordinate) and eventually outside of a fixed 
ball ${B}$ containing
 all of the $f^{\inv}W_i$.    The Thom form of the translated copy of $f^{-1}W_k \subset \R^{d}$ can
then be cobounded by a form whose support is outside of ${B}$, so   $d^{\inv} f^*\omega_k$ is the
sum of this cobounding form outside of ${B}$ along with a Thom form of $E_{k}$.

Because $f^{*} \omega_{k-1}$ has support in $B$, its wedge product with
$ \dinv f^{*}\omega_{k}$ is equal to its wedge product with the Thom form  of $E_{k}$.  
Because the wedge product of Thom forms is a Thom form of the intersection,
this will be a Thom form of the submanifold of 
$f^{-1} W_{k-1}$ of points which lie above (in $\R^d$) 
some point of $f^{-1} W_{k}$.  
Proceeding in this fashion,
the Hopf invariant of $f$ will in the end
be the generic count of collections of a point in $f^{-1} W_{1}
\subset \R^{d}$ which
lies above a point in $f^{-1} W_{2}$ which in turn lies above a point in $f^{-1} W_{3}$, etc.

The resulting count is a generalized linking number, equal to the degree of the Gauss map
from $\prod_{k} f^{-1} W_{i}$ to $\prod_{k-1} S^{d-1}$ sending $(x_{1}, \ldots, x_{k})$
to $(\cdots, \frac{x_{i} - x_{i-1}}{||x_{i} - x_{i-1}||}, \cdots)$.   This Gauss map factors
through the universal Gauss map sending  $\prod f^{-1}(W_{i})$ to the 
ordered configuration space  ${\rm Conf}_{k}(\R^{d})$ by sending a collection
of points with one point in each $f^{-1} W_{i}$ to the configuration given by their
images in $\R^{d}$.  These degrees  
are encoding the image in homology of the  fundamental class  of 
$\prod f^{-1}(W_{i})$ included in ${\rm Conf}_{k}(\R^{d})$.  As we discuss in 
Section~\ref{S:compare}, these Hopf invariants arising from Thom classes of
disjoint submanifolds agree with some of those defined by Boardman and Steer
\cite{BoSt66}, as applied by Koschorke \cite{Kosh97}.

\medskip

If the submanifolds $W_{i}$ are not disjoint, then their linking numbers need to include
``correction terms'' given by submanifolds which bound their intersections.  The simplest
example is that of some $W_{1} \cap W_{2} = \partial T$.  Then there is a cocycle in the bar
complex of the form $\omega_{1} | \omega_{2} \pm \theta$, where as before $\omega_{i}$
are Thom cochains of the $W_{i}$ and now $\theta$ is a Thom cochain of $T$ so that $d\theta
= \omega_{1} \wedge \omega_{2}$.  

This cocyle in the bar complex is a special case of Example~\ref{E:arbwt1}.  The associated Hopf invariant is 
$ \int_{S^{d_{1} + d_{2} - 1}} {\dinv f^{*} \omega_{1} \wedge f^{*} \omega_{2} \mp f^{*}\theta},$
where the $d_{i}$ are the codimensions of the $W_{i}$.
An individual $f$ may be assumed to be transverse to $W_{1}$ and $W_{2}$, in which case their preimages
are disjoint, and the first part of this integral is their linking number.  But through a homotopy $H$ from $f$
to another map the preimages of the $W_{i}$ may intersect at some time.  

The change in linking
number which occurs because of this intersection
is accounted for by a corresponding change in $\int_{S^{d_{1} + d_{2} -1}} f^{*} \theta$, which counts the
preimages of $T$.    See Figure~\ref{F1}, in which for $f$ (that is, $H|_{t=0}$) the Hopf invariant has a contribution
of $\pm 1$ from  $ \int_{S^{d_{1} + d_{2} - 1}} {\dinv f^{*} \omega_{1} \wedge f^{*} \omega_{2}}$
and a contribution of $\pm 1$ from  $\int_{S^{d_{1} + d_{2} -1}} f^{*} \theta$, but the Hopf invariant of 
the $H|_{t=1}$ has two contributions of $\pm 1$ of the latter form, counting preimages of $T$.

\begin{figure}[ht]\label{F1}

\def\svgwidth{8.5cm}
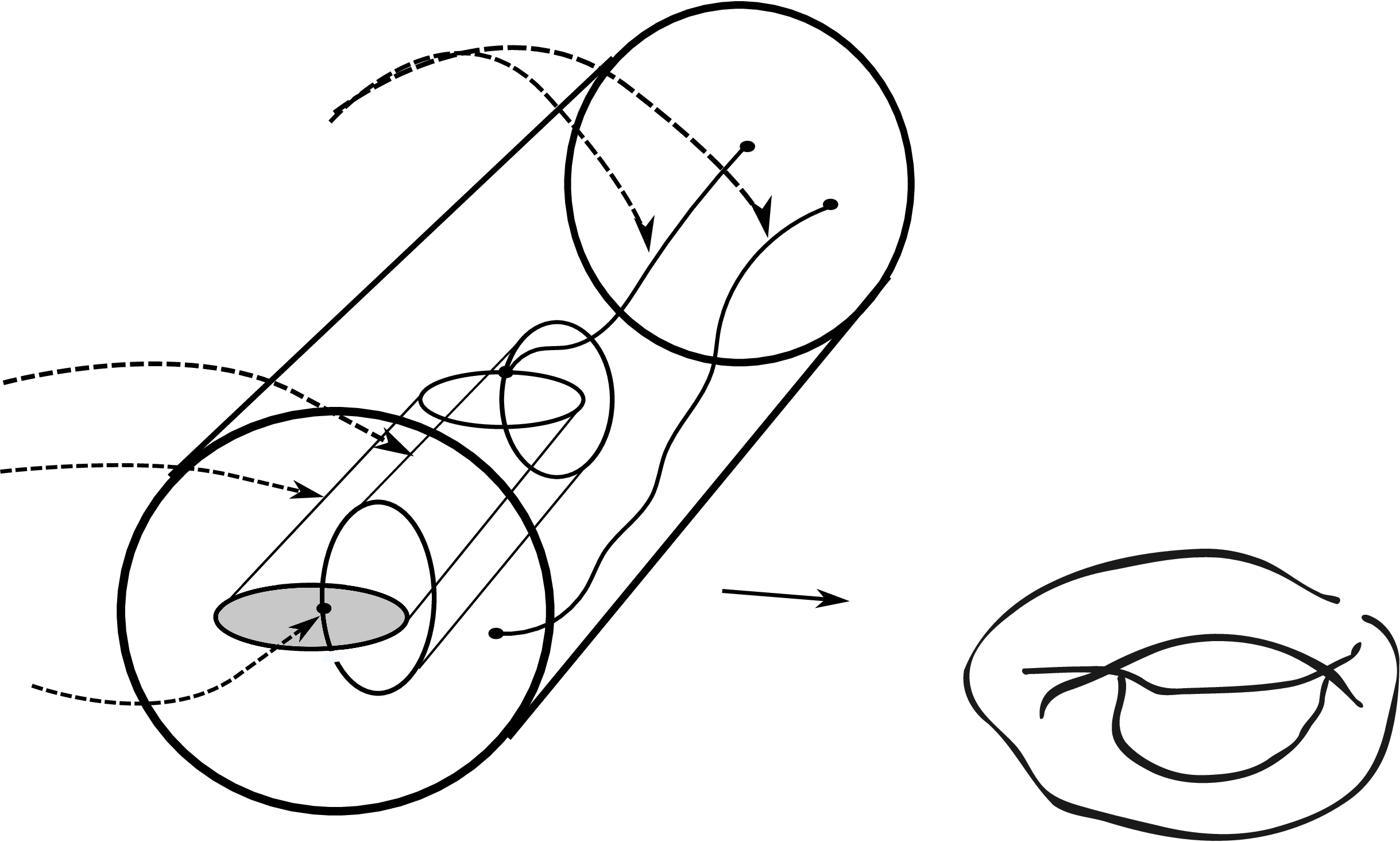
\caption{An illustration of ``linking with correction'' remaining constant through a homotopy.}
\end{figure}

All of the formulae we gave in the Section~\ref{examples} 
may be interpreted in the language of intersection, 
linking, and boundary behavior of submanifolds, when the classes
in question are Thom classes (which they may always be assumed to be
if we allow submanifolds with singularities).
Indeed, generalizing the definition of Hopf invariants through linking numbers
was the starting point for our project.  It was in elucidating this geometry that 
we were led to the formalism of Lie coalgebras, with the cohomology of 
configuration spaces playing a central role.

\subsection{The generalized Hopf invariant one question}

From \refP{spacelevel} and the adjoint of the Milnor-Moore theorem 
that 
$\pi_{*}(\Omega X) \otimes \Q \to H_{*}(\Omega X; \Q)$ is injective \cite{MiMo65},
we have the following.

\begin{proposition}\label{P:hopfinvar}
If $X$ is simply connected,
the map $\eta_{B} : H_{B}^{*-1}(X; \Q) \to {\rm Hom}(\pi_{*}(X), \Q)$ is surjective.
In particular, the map $\eta_{B} : H_{B}^{*-1}(X) \to {\rm Hom}(\pi_{*}(X), \Z)$ is 
full rank.
\end{proposition}

The question of the
kernel of this map, which is large, is addressed in the next sections.  
That such a large kernel arises is
explained from the operadic viewpoint as the fact that we are taking the wrong 
bar construction.  The rational PL cochains on a simplicial set are commutative,
so we should be taking a bar construction over the Koszul dual cooperad, namely the 
Lie cooperad, rather than associative cooperad.  That is, the correct homology theory
for a commutative 
algebra is Andr\'e-Quillen homology.

The question of the cokernel of the integral map is 
a natural generalization of one of the most famous questions in the
history of topology, namely the
Hopf invariant one problem.
The cokernel is trivial immediately for $X$ an odd sphere and is $\Z/2$ for $X$ an even
sphere other than $S^{2}$, $S^{4}$ and $S^{8}$ by Adams' celebrated result.
We will indicate a new line of attack on this problem after we develop
the Lie coalgebraic approach.

\subsection{Evaluation on Whitehead products}

Some properties of Hopf invariants, such as naturality, are immediate.  
A deeper result is to use the standard coproduct of $B(X)$ to compute 
the value of our generalized Hopf invariants on Whitehead products.

Let $f:S^n \to X$ and $g:S^m \to X$, and recall the definition of the Whitehead product: 
$$[f,g]:S^{n+m-1} \xrightarrow{\ W\ } S^n \vee S^m \xrightarrow{\ f \vee g\ } X.$$
We have $(f\vee g)^* = f^* + g^*$ on $C^{*}(X)$, so on $B(X)$ the map $(f\vee g)^*$ 
acts by $(f^* + g^*)$ on each component of a bar expression.  For example, 
\begin{align*}
 (f\vee g)^* \bigl(\s x_1\;|\; \s x_2\bigr) &= \s (f^* + g^*)x_1 | \s (f^* + g^*) x_2  \\
   &= \s f^* x_1 | \s f^* x_2\,  +\,  \s f^* x_1 | \s g^* x_2 
     \, +\, \s g^* x_1 | \s f^* x_2\, +\, \s g^* x_1 | \s g^* x_2
\end{align*}
We will generally mean for $f^*$ and $g^*$ to be considered as 
maps $B(X) \to B(S^n \vee S^m)$ as in this example, omitting the inclusions of
and projections onto wedge factors.

\begin{theorem}\label{T:cobracket}
 Let $\gamma$ be a cocycle in $B^{n+m-2}(X)$ and $f:S^n\to X$, $g:S^m\to X$.  Then 
  $$\left\langle \gamma,\ [f, g] \right\rangle_\eta = 
       \sum_j \langle \alpha_j,\; f\rangle_\eta \cdot \langle \beta_j,\; g\rangle_\eta  
          \ \mp \ \langle \alpha_j,\ g\rangle_\eta \cdot \langle \beta_j,\ f\rangle_\eta$$
  where $\Delta \gamma\ \simeq \sum_j \alpha_j \otimes \beta_j$, 
  with all $\alpha_j$ and $\beta_j$ closed, and $\mp$ is minus the Koszul sign induced
  by moving $s^\inv \alpha_j$ past $s^\inv \beta_j$.
\end{theorem}

Our convention is for
$\langle \gamma,\ f\rangle_\eta = 0$ if $f:S^n \to X$ and $|\gamma| \neq n-1$.
The relation $\Delta\gamma \simeq \sum_j \alpha_j\otimes \beta_j$ with $\alpha_i$
and $\beta_i$ closed 
follows from the fact that $\Delta \gamma$ is closed in $B(X)\otimes B(X)$ 
and application of the K\"unneth isomorphism. 

In the next section, we will see that Hopf invariants completely determine the homotopy
Lie algebra of $X$ rationally.  The present theorem recovers some of the integral
information as well.

\begin{proof}

Because  $\langle\gamma,\ [f,g]\rangle_\eta = \langle (f\vee g)^*\gamma,\ W\rangle_\eta$
we may compute the Hopf invariant in $S^{n+m-1}$ by doing intermediate work in
$S^n\vee S^m$.
Recall $\bar C^{*}(S^n\vee S^m) \cong \bar C^*(S^n)\oplus\bar C^*(S^m)$,
which induces a component-bigrading on 
$B (S^n\vee S^m)$ as follows. 
If $\mu$ is a monomial in  $B(S^n\vee S^m)$, we say that $w \in B(S^{n})$ 
is an $S^n$ component of $\mu$ if it is a maximal length subword
with support only in $S^n$; we define $S^m$ components of $\omega$ similarly.  
Bigrade $B(S^n\vee S^m)$ by the 
number of $S^n$ components and the number of $S^m$ components.   
If $a,b\in C^{*}(S^n\vee S^m)$ have disjoint supports then $a\cup b = 0$, so
the differential of $B(S^n\vee S^m)$ preserves component-bigrading.

For example, let $a_i, b_i \in C^*(S^n\vee S^m)$ where the $a_i$ have support
in $S^n$ and the $b_i$  in $S^m$.  Then 
$\mu =  \s  \s b_1 | \s a_2 | \s a_3 | \s b_2 | \s b_3 | \s b_4$ has 
$S^n$ component  $\s a_2|\s a_3$  and $S^m$ components
$\s b_1$ and $\s b_2 | \s b_3 | \s b_4$, so $\mu$ has component bigrading
$(1,2)$.

Given a cocycle $\gamma \in B^{n+m-2}(X)$, the pullback 
$(f\vee g)^*\gamma \in B(S^n\vee S^m)$ splits as a sum
of cocyles in each component-bigrading. By the K\"unneth theorem,
$B(S^n\vee S^m)$ has no homology in degree $n+m-2$ away from bigrading 
$(1,1)$  and possibly $(1,0)$ and $(0,1)$.  So all terms of $(f\vee g)^*\gamma$ 
not in these bigradings are exact.  The terms in bigrading $(1,0)$ and $(0,1)$
are $f^*\gamma$ and $g^*\gamma$ which are exact after being pulled back 
 to $S^{n+m-1}$
by $W$.  Thus it suffices to focus on just the terms of $(f\vee g)^* \gamma$
in bigrading $(1,1)$.

Bigrading $(1,1)$ splits as the sum of two subcomplexes -- one where 
the $S^n$ component comes first, and the other where the $S^m$ component 
comes first.  The subcomplex of bigrading $(1,1)$ with the $S^n$ component 
first is isomorphic to $B(S^n) \otimes B(S^m)$. 
The terms of $(f\vee g)^*\gamma$ in this subcomplex are 
$$\sum_{\Delta\gamma = \sum_i a_i \otimes b_i} f^*(a_i) | g^*(b_i).$$ 
If $\sum_i a_i \otimes b_i \simeq \sum_j \alpha_j \otimes \beta_j$ then the 
cobounding expression in $B(X)\otimes B(X)$ will determine a homotopy between
the above terms and
$\sum_j f^*(\alpha_j) | g^*(\beta_j).$
By the K\"unneth theorem this is cohomologous to 
$\sum_j \tau(f^*\alpha_j) | \tau(g^*\beta_j).$
By our work in example \ref{EX:hopf on brackets} the Hopf pairing of this is
$\sum_j \langle\alpha_j,\ f\rangle_\eta \cdot \langle\beta_j,\ g\rangle_\eta$
as desired.

The terms of $(f\vee g)^*\gamma$ in bigrading $(1,1)$ with the $S^m$ component
first are similar, but with a shift of sign at the final step due to the change
in orientation induced by the isomorphism $D^n \times D^m \cong D^m \times D^n$.
\end{proof}

This theorem implies for example that the Hopf invariant of a cocycle of the 
form $\alpha|\beta \mp \beta|\alpha$ evaluates trivially on all Whitehead products.
Indeed, we saw in \refE{arbwt1} that this Hopf invariant is zero.
There are many further classes of Hopf invariants which vanish, though proving
this through direct methods quickly becomes difficult.  However \refT{cobracket}
suggests a key for identifying bar expressions with vanishing Hopf invariants.
If we define in the obvious way an ``anti-commutative coproduct'' on the bar construction,
the submodule where this coproduct vanishes will by \refT{cobracket} have Hopf
invariants which evaluate trivially on Whitehead products.  
Quotienting by this submodule yields the Lie coalgebraic cobar construction, as 
discussed in Section~3 of \cite{SiWa06}. This is our focus for the rest of the paper.

\section{Hopf invariants from the Lie coalgebra model of a space, and their completeness}

We now switch to the rational, commutative 
setting and we replace
the bar construction $B$ with the Lie coalgebraic bar construction $\E$ from \cite{SiWa06}.

\subsection{Lie coalgebraic bar construction}

\begin{definition}\label{D:eiln}
Let $V$ be a vector space. 
Define $\cofreeE(V)$ to be the quotient of 
$\mathbb{G}(V)$ by ${\rm Arn}(V)$, where $\mathbb{G}(V)$
is the span of the set of oriented acyclic graphs with vertices
labeled by elements of $V$ modulo multilinearity in the vertices,
and where ${\rm Arn}(V)$ is the subspace generated by
 arrow-reversing and Arnold expressions:
\begin{align*}
\text{(arrow-reversing)}\qquad & \qquad
\begin{xy}                           
  (0,-2)*+UR{\scriptstyle a}="a",    
  (3,3)*+UR{\scriptstyle b}="b",     
  "a";"b"**\dir{-}?>*\dir{>},         
  (1.5,-5),{\ar@{. }@(l,l)(1.5,6)},
  ?!{"a";"a"+/va(210)/}="a1",
  ?!{"a";"a"+/va(240)/}="a2",
  ?!{"a";"a"+/va(270)/}="a3",
  "a";"a1"**\dir{-},  "a";"a2"**\dir{-},  "a";"a3"**\dir{-},
  (1.5,6),{\ar@{. }@(r,r)(1.5,-5)},
  ?!{"b";"b"+/va(90)/}="b1",
  ?!{"b";"b"+/va(30)/}="b2",
  ?!{"b";"b"+/va(60)/}="b3",
  "b";"b1"**\dir{-},  "b";"b2"**\dir{-},  "b";"b3"**\dir{-},
\end{xy}\ +\ 
\begin{xy}                           
  (0,-2)*+UR{\scriptstyle a}="a",    
  (3,3)*+UR{\scriptstyle b}="b",     
  "a";"b"**\dir{-}?<*\dir{<},         
  (1.5,-5),{\ar@{. }@(l,l)(1.5,6)},
  ?!{"a";"a"+/va(210)/}="a1",
  ?!{"a";"a"+/va(240)/}="a2",
  ?!{"a";"a"+/va(270)/}="a3",
  "a";"a1"**\dir{-},  "a";"a2"**\dir{-},  "a";"a3"**\dir{-},
  (1.5,6),{\ar@{. }@(r,r)(1.5,-5)},
  ?!{"b";"b"+/va(90)/}="b1",
  ?!{"b";"b"+/va(30)/}="b2",
  ?!{"b";"b"+/va(60)/}="b3",
  "b";"b1"**\dir{-},  "b";"b2"**\dir{-},  "b";"b3"**\dir{-},
\end{xy} \\
\text{(Arnold)}\qquad & \qquad
\begin{xy}                           
  (0,-2)*+UR{\scriptstyle a}="a",    
  (3,3)*+UR{\scriptstyle b}="b",   
  (6,-2)*+UR{\scriptstyle c}="c",   
  "a";"b"**\dir{-}?>*\dir{>},         
  "b";"c"**\dir{-}?>*\dir{>},         
  (3,-5),{\ar@{. }@(l,l)(3,6)},
  ?!{"a";"a"+/va(210)/}="a1",
  ?!{"a";"a"+/va(240)/}="a2",
  ?!{"a";"a"+/va(270)/}="a3",
  ?!{"b";"b"+/va(120)/}="b1",
  "a";"a1"**\dir{-},  "a";"a2"**\dir{-},  "a";"a3"**\dir{-},
  "b";"b1"**\dir{-}, "b";(3,6)**\dir{-},
  (3,-5),{\ar@{. }@(r,r)(3,6)},
  ?!{"c";"c"+/va(-90)/}="c1",
  ?!{"c";"c"+/va(-60)/}="c2",
  ?!{"c";"c"+/va(-30)/}="c3",
  ?!{"b";"b"+/va(60)/}="b3",
  "c";"c1"**\dir{-},  "c";"c2"**\dir{-},  "c";"c3"**\dir{-},
  "b";"b3"**\dir{-}, 
\end{xy}\ + \                             
\begin{xy}                           
  (0,-2)*+UR{\scriptstyle a}="a",    
  (3,3)*+UR{\scriptstyle b}="b",   
  (6,-2)*+UR{\scriptstyle c}="c",    
  "b";"c"**\dir{-}?>*\dir{>},         
  "c";"a"**\dir{-}?>*\dir{>},          
  (3,-5),{\ar@{. }@(l,l)(3,6)},
  ?!{"a";"a"+/va(210)/}="a1",
  ?!{"a";"a"+/va(240)/}="a2",
  ?!{"a";"a"+/va(270)/}="a3",
  ?!{"b";"b"+/va(120)/}="b1",
  "a";"a1"**\dir{-},  "a";"a2"**\dir{-},  "a";"a3"**\dir{-},
  "b";"b1"**\dir{-}, "b";(3,6)**\dir{-},
  (3,-5),{\ar@{. }@(r,r)(3,6)},
  ?!{"c";"c"+/va(-90)/}="c1",
  ?!{"c";"c"+/va(-60)/}="c2",
  ?!{"c";"c"+/va(-30)/}="c3",
  ?!{"b";"b"+/va(60)/}="b3",
  "c";"c1"**\dir{-},  "c";"c2"**\dir{-},  "c";"c3"**\dir{-},
  "b";"b3"**\dir{-}, 
\end{xy}\ + \                              
\begin{xy}                           
  (0,-2)*+UR{\scriptstyle a}="a",    
  (3,3)*+UR{\scriptstyle b}="b",   
  (6,-2)*+UR{\scriptstyle c}="c",    
  "a";"b"**\dir{-}?>*\dir{>},         
  "c";"a"**\dir{-}?>*\dir{>},          
  (3,-5),{\ar@{. }@(l,l)(3,6)},
  ?!{"a";"a"+/va(210)/}="a1",
  ?!{"a";"a"+/va(240)/}="a2",
  ?!{"a";"a"+/va(270)/}="a3",
  ?!{"b";"b"+/va(120)/}="b1",
  "a";"a1"**\dir{-},  "a";"a2"**\dir{-},  "a";"a3"**\dir{-},
  "b";"b1"**\dir{-}, "b";(3,6)**\dir{-},
  (3,-5),{\ar@{. }@(r,r)(3,6)},
  ?!{"c";"c"+/va(-90)/}="c1",
  ?!{"c";"c"+/va(-60)/}="c2",
  ?!{"c";"c"+/va(-30)/}="c3",
  ?!{"b";"b"+/va(60)/}="b3",
  "c";"c1"**\dir{-},  "c";"c2"**\dir{-},  "c";"c3"**\dir{-},
  "b";"b3"**\dir{-},
\end{xy}
\end{align*}
Here ${a}$, ${b}$, and ${c}$ are elements of $V$ labeling vertices of a graph
which could possibly have edges connecting to
other parts of the graph (indicated by the ends of edges abutting 
$a$, $b$, and $c$),
which are not modified in these operations. 

A Lie cobracket is
defined as $$] G [\ = \sum_{e \in G} (G^{\hat{e}}_1 \otimes G^{\hat{e}}_2 - 
G^{\hat{e}}_2 \otimes G^{\hat{e}}_1),$$ 
where $e$ ranges
over the edges of $G$, and $G^{\hat{e}}_1$ and $G^{\hat{e}}_2$ 
are the connected components
of the graph obtained by removing $e$, which points 
to $G^{\hat{e}}_2$.   In \cite{SiWa06} we show this is well defined.  For example,
\begin{align*}
\left]\graphmp{a}{b}{c}
\right[\ =&   
\left( \linep{b}{c}\otimes 
       \bullet^{a}\right) -
\left( \bullet^{a}\otimes 
       \linep{b}{c}\right) \\
  &+  \left( \linep{b}{a}\otimes 
       \bullet^{c}\right)  -  
\left( \bullet^{c} \otimes 
       \linep{b}{a}\right).
\end{align*}
\end{definition}

We require a graded version of $\cofreeE(V)$ which
we define in Section~3 of \cite{SiWa06}.   Our graph-theoretic representation
is an explicit model for Lie coalgebras, as first studied in \cite{Mich80}.

\begin{proposition}
If $V$ is graded in positive degrees, then
$\cofreeE(V)$ is isomorphic to the cofree Lie coalgebra on $V$.
\end{proposition}

Indeed, if $V$ and $W$ are linearly dual, then the {\em configuration pairing} of 
\cite{Sinh06.2, SiWa06, Walt10} can be used to define a perfect pairing between $\cofreeE(V)$ and the
free Lie algebra on $W$.

\begin{definition}\label{D:E}
Let $A$ be a one-connected commutative differential graded algebra 
with differential $d_{A}$ and  multiplication 
$\mu_{A}$.  Define $\mathcal{G}(A)$ to be  the total complex of the bicomplex
$
 \bigl(\mathbb{G}(s^{\inv}\bar A),\ 
  d_{A},\ d_\mu\bigr).$
Here $s^{\inv}\bar A$ is the desuspension of the ideal of positive-degree elements of $A$,
the ``internal'' differential $d_A$ is given by cofreely extending that of $A$ by the Leibniz rule, 
and the ``external'' differential
$d_\mu (g) = 
  \sum_{e} \mu_e (g)$ where up to sign $\mu_{e} (g)$ contracts the edge $e$ in $g$,
  multiplying the elements of $A$ labeling its endpoints to obtain the new label. 
  (For the sign convention see Definition~4.5 of \cite{SiWa06}.)
Let $\E(A)$ be the quotient $\mathcal{G}(A)\diagup\text{Arn}(s^{\inv}\bar A)$.  
  
We define internal degree and weight of monomials in $\E(A)$ as the sum of the
degrees of component elements and the number of vertices in the graph monomial.
As before, total degree is internal degree minus the weight.

Define $\E(X)$ to be $\E(A^*(X))$ where $A^*(X)$ is  a
model for rational commutative cochains 
on $X$, and let $H_{\E}^{*}(X)$ denote $H^{*}(\E(X))$.
\end{definition}

The complex  $\E(A)$ is the ``Lie coalgebraic'' bar 
construction on a commutative algebra which computes Harrison homology, 
constructed presently as a quotient of ``graph coalgebraic''
bar construction $\mathcal{G}(A)$.

Harrison homology is a special case of
Andr\'e-Quillen homology \cite{Andr74, Quil70}
for one-connected commutative differential graded algebras.   
Recall that the Harrison chains $B_{H}(A)$ 
are constructed dually to Quillen's functor $\mathcal{L}$.  The bar construction of a commutative
algebra is a Hopf algebra under the shuffle product and the splitting coproduct.  The 
Harrison complex is given by quotienting to Hopf algebra indecomposables.  By a celebrated
theorem of Barr \cite{Barr68}, the quotient map $p:BA \to B_{H}A$ has a splitting $e$ 
when working over a field of characteristic 0.

\begin{proposition}\label{P:harriso}
There is a short-exact sequence of bicomplexes, giving an isomorphism of final terms
$$
\begin{CD}
0 @>>> {\rm{Sh}}(A) @>>> B(A) @>>> B_H(A) @>>> 0 \\
&&  @VVV @V{\phi}VV @V{\cong}VV \\
0 @>>> {\rm Arn}(A) @>>> \mathcal{G}(A) @>>> \E(A) @>>> 0,
\end{CD}
$$
where $\phi$ sends the bar expression $a_1 | a_2 | \cdots | a_n $ to 
the graph  \    
$\begin{xy}
 (-6,-2)*+UR{\scriptstyle a_1}="1",
 (0,2)*+UR{\scriptstyle a_2}="2",
 (6,-2)*+UR{\cdots}="3",
 (12,2)*+UR{\scriptstyle a_n}="4",
 "1";"2"**\dir{-}?>*\dir{>},
 "2";"3"**\dir{-}?>*\dir{>},
 "3";"4"**\dir{-}?>*\dir{>}
\end{xy}$.
Furthermore this is an isomorphism of Lie coalgebras 
when the Harrison complex is given the Lie coalgebra structure defined
by Schlessinger and Stasheff \cite{ScSt85}.  
\end{proposition}

Our presentation of Harrison homology via graphs has a critical advantage over the 
classical construction.  While the 
set of generators of  $\mathcal{G}(A)$ is larger than that of $B(A)$, 
the set of relations $\text{Arn}(A)$ are simpler to express 
and are defined locally, unlike the shuffle relations $\text{Sh}(A)$.  
This greatly simplifies some proofs, constructions, and calculations.
By abuse we will sometimes use bar notation to refer to elements 
of $\E(A)$, suppressing $\phi$. 

The significance of Harrison homology in this setting is that it bridges the worlds
of cohomology and homotopy.   Hopf invariants give a  geometric understanding of this bridge.  
For convenience, we use Barr's splitting and our development of Hopf
invariants for the standard bar complex to quickly establish the basic properties.
Deeper properties and explicit constructions require the Lie coalgebraic formulation.  

\subsection{Hopf invariants}

As we did for the bar complex,  we start with the fact that 
$H^{n-1}_{\E}(S^n)$ is rank one, generated by an element of weight one.

\begin{definition}
Given a cocycle $\gamma \in \E^{n-1}S^n$ we let $\tau(\gamma) \simeq \gamma$ be any 
cohomologous cocyle of weight one.

Write $\int_{\E(S^n)}$ for the map from cocycles in $\E^{n-1}(S^n)$ to $\mathbb{Q}$ 
given by $\int_{\E(S^n)} \gamma = \int_{S^n} \tau(\gamma)$.
\end{definition}

\begin{lemma}\label{L:Ezero}
The map $\int_{\E(S^n)}$ is well defined and induces the isomorphism $H^{n-1}_{\E}(S^{n}) \cong \Q$.
\end{lemma}

\begin{proof}
We observe that $\int_{\E(S^{n})} = \int_{B(S^{n})} \circ e$, since Barr's splitting
$e$ is an isomorphism
which is an ``identity'' on the weight one parts of these complexes.  Thus this lemma follows from 
Lemma~\ref{L:key}.
\end{proof}

Following \cite{SiWa06}, it would be better to use
the graph coalgebraic bar construction $\mathcal{G}(A)$ here, but we chose to use the classical
bar construction because it is more familiar, has an established Barr splitting, and is directly tied to cohomology of loopspaces.

We can now take our Definition~\ref{D:hopfform} of Hopf invariants
and replace the bar complex everywhere by
$\E$ (and the integers by the rational numbers).  When necessary, we distinguish 
the Hopf pairings and Hopf invariants by decorations such as $\langle \;,\; \rangle_{\eta}^{\E}$
vs. $\langle \;,\;\rangle_{\eta}^{B}$,
or $\eta^{\E}$ vs. $\eta^{B}$, etc.  

\begin{theorem}\label{T:splitcompat}
Hopf invariants are compatible with the quotient map $p$ and  Barr's splitting $e$
between $B(X)$ and $\E(X)$.  That is, $\eta^{B} = \eta^{\E} \circ p$ and
$\eta^{\E} = \eta^{B} \circ e$.
\end{theorem}

\begin{proof}
We consider the adjoint maps and fix an $f \in \pi_{n}(X)$. The images
of $(\eta^{B})^{\dagger}(f)$ and $(\eta^{\E})^{\dagger}(f)$ are then given by the composites
going from left to right in the following diagram. 
$$
\xymatrix@R=7pt@C=50pt{
 B(X) \ar[r]^{f^*} \ar@<-2pt>[dd]_{p} & B(S^n) \ar@<-2pt>[dd]_{p} \ar[dr]^{\int_{B(S^n)}} & \\
 & & \mathbb{Q} \\ 
 \E(X)  \ar[r]^{f^*} \ar@<-2pt>[uu]_{e} & \E(S^n) \ar@<-2pt>[uu]_{e} \ar[ur]_{\int_{\E(S^n)}} & 
}
$$
The result follows  from commutativity of the two squares (one involving $p$ and
one involving $e$) and two
triangles in the above diagram. 
The squares commute due to naturality of the quotient map and Barr's splitting, and 
the triangles commute by
our comment in the proof of Lemma~\ref{L:Ezero}.
\end{proof}

We  now see that all of the ``shuffles'' in the kernel of the quotient 
map $p : B(X) \to \E(X)$ give rise to the zero homotopy functional, constituting 
a large kernel.

Using this theorem, we can translate some theorems from the bar setting, such as 
the following which is immediate
from \refT{cobracket}.

\begin{corollary}\label{C:cobracket}
 Let $\gamma$ be a cocycle in $\E^{n+m-2}(X)$ and $f:S^n\to X$, $g:S^m\to X$. 
 Then   

 $$\left\langle \gamma,\ [f, g]\right\rangle^\E_\eta = 
	   \sum_j \langle \alpha_j,\; f\rangle^\E_\eta \cdot \langle \beta_j,\; 
	   g\rangle^\E_\eta$$
  where $]\gamma[\ \simeq \sum_j \alpha_j \otimes \beta_j$, 
  with all $\alpha_j$ and $\beta_j$ closed.  
\end{corollary}

By recasting our Hopf invariants  in terms of Lie coalgebras, we get some 
immediate applications.

\begin{example}\label{E:wedge1}
Consider a wedge of spheres $X = \bigvee_{i} S^{d_{i}}$.  Let $\omega_{i}$ represent
a generator of $H^{d_{i}}(S^{d_{i}})$ and $W$ be the differential graded vector space
spanned by the $\omega_{i}$ with trivial products and differential.
Then $W$ is a model for cochains on $X$  and $\E(W) = {\mathbb{E}}(W)$ is just the cofree Lie coalgebra
on $W$.

\refT{cobracket} along with Proposition~3.6 of \cite{SiWa06} imply that the value of some Hopf
invariant $\eta_{\gamma}$ on an iterated Whitehead product $P$ of the basic inclusions
$\iota_{i} = [S^{d_{i}} \hookrightarrow X] \in \pi_{d_{i}}(X)$ is given by the configuration pairing (defined in Corollary~3.5 of \cite{SiWa06}) of $\gamma$ and $P$.  
This immediately implies that the $\iota_{i}$ generate a copy of a free Lie algebra in the 
rational homotopy groups of $X$.  

Because this example is universal, a corollary to \refT{cobracket} is that in general the value of
a Hopf invariant on an iterated Whitehead product can be calculated through a configuration pairing. 
\end{example}

\subsection{Completeness}

Hopf invariants associated to $\E(X)$ give a complete and sharp picture of rational homotopy
groups.

\begin{theorem}\label{T:complete}
If $X$ is simply connected, then $\eta^\E : H^{*-1}_\E(X) \to {\rm Hom}( \pi_{*}(X),\; \Q)$ is an isomorphism.
\end{theorem}

There are three possible lines of proof.  Using \refT{cobracket} and Example~\ref{E:wedge1}
we could show that $\eta$ agrees with 
the isomorphism given by the Quillen equivalences established in \cite{SiWa06}. 
Alternately we could try to start with the surjectivity of this map, which follows
from \refP{hopfinvar}, and try to calculate the kernel.
Instead we shall
prove more directly that $\eta$ is an isomorphism by developing
a long exact sequence of a fibration for $\E$.
This approach yields an independent proof and gives a more direct geometric understanding
of the isomorphism in \refT{complete}.

\begin{definition}
Let $\E(E, F)$ be the kernel of the map from $\E(E)$ to $\E(F)$ induced by inclusion.
Define relative Hopf invariants for $\pi_{n}(E, F)$, by pulling back a cycle
in $\E(E, F)$ to $\E(D^{n},S^{n-1})$ and then reducing to weight one and evaluating.
\end{definition}

That $H^{n-1}_{\E}(D^{n}, S^{n-1})$ is rank one generated by a cycle of weight one 
follows from a direct calculation.
To see that relative Hopf invariants are well defined 
we may use the same weight-reduction argument.

By construction, there is a long exact sequence for $H^{*}_{\E}(F)$, $H^{*}_{\E}(E)$ 
and $H^{*}_{\E}(E, F)$.  Hopf invariants give a pairing between between
this long exact sequence and the long exact sequence for relative homotopy groups.
For clarity, we will drop the identifier $\eta$ from our Hopf pairings below 
and instead subscript them by their ambient space.

\begin{proposition}
The following long exact sequences pair compatibly
$$\xymatrix@R=5pt{
\cdots \ar[r] & \pi_{n}(F) \ar[r]^{i_{*}} & \pi_{n}(E) \ar[r]^{j_{*}} & \pi_{n}(E, F) 
		\ar[r]^{\partial} & \pi_{n-1}(F) \ar[r] & \cdots \\
 \cdots &  H^{n-1}_{\E}(F) \ar[l] & H^{n-1}_{\E}(E) \ar[l]_{i^{*}} & H^{n-1}_{\E}(E, F) \ar[l]_{j^{*}} & H^{n-2}_{\E}(F) 
		\ar[l]_{\delta} & \cdots \ar[l]
}$$
That is, for $f \in \pi_{n}(F)$ and $\gamma \in H^{n-1}_{\E}(E)$, $\langle i^{*} \gamma, f \rangle_{F} =
\langle \gamma, i_{*} f \rangle_{E}$, and similarly elsewhere.
\end{proposition}

\begin{proof}

For $i_{*}$ and $i^{*}$ this is immediate from naturality, 
as both $\langle i^{*} \gamma,\, f \rangle_{F} $ and
$\langle \gamma,\, i_{*} f \rangle_{E}$ are equal to 
$\int_{S^{n}} \tau\left((f \circ i)^*\gamma\right)$. The 
argument for $\langle j^{*} \gamma,\, f \rangle_{E} $  and 
$\langle \gamma,\, j_{*} f \rangle_{(E, F)}$ 
is the same
once one substitutes $\pi_{n}(E, pt.)$ for $\pi_{n}(E)$ and $H^{n}_{\E}(E, pt.)$ for 
$H^{n}_{\E}(E)$.  In this case one now sees both pairings as given by 
$\int_{D^{n}} \tau\left((f \circ j)^*\gamma\right)$.

Finally we consider the connecting homomorphisms. 
Recall that $\delta : H^{n-2}_{\E}(F) \to H^{n-1}_{\E}(E, F)$ 
takes a cocycle $\gamma\in \E^{n-2}(F)$, lifts it at the cochain level to $\E^{n-2}(E)$ and then takes
its coboundary.  Thus given $f:(D^n, S^{n-1}) \to (E,F)$, we have
$\langle \delta(\gamma),\, f\rangle_{(E,F)} = \int_{D^{n}} \tau(f^*d_{\E} \bar{\gamma})$, 
where $\bar{\gamma}$ is a lift of $\gamma$ to $E$.  
Using the same lift $\bar{\gamma}$ we see that, 
$$\langle \gamma,\, \partial f \rangle_{F} =  \int_{S^{n-1}} \tau\bigl( (\partial f)^* \gamma\bigr) 
= \int_{\partial D^{n}} \tau(f^*\bar{\gamma}) = \int_{D^n} d\, \tau(f^*\bar{\gamma}).$$
\end{proof}

We now connect with the exact sequences of a fibration.
Let $p: E \to B$ be a fibration and write $\lambda : \pi_{n}(B) \to \pi_{n}(E, F)$ for the standard
map defined by lifting.

\begin{proposition}\label{P:E_fibration}
The map $p^{*} : \E(B) \to \E(E, F)$ is a weak equivalence. 
\end{proposition}

\begin{corollary}
Let $\gamma$ be a cycle in $\E(E,F)$.  Then $\gamma$ is homologous to $p^*(\bar\gamma)$ for
some $\bar\gamma$ in $\E B$.
\end{corollary}

\begin{proof}[Proof of Proposition \ref{P:E_fibration}]
We use the maps induced by the Sullivan model for the fiber inclusion map $i$,
as given in \cite{FHT01} \S 15(a):
$$
\xymatrix{
  \apl F & \apl E  \ar[l]_(.57){i^*} &  \apl B \ar[l]_(.47){p^*} \ar[dl]^(.47){\iota} \\
 (\Lambda V_F, \bar d) \ar[u]^{m_F}_{\simeq} &
   (\apl B \otimes \Lambda V_F, d). \ar[l]_(.57){\varepsilon} \ar[u]^{ m}_{\simeq} &
    & } 
$$
Here $(\Lambda V_F, \bar d)$ is the Sullivan model for $\apl F$, $\iota$ is the map induced 
by the inclusion of the unit $\Q \to \Lambda V_F$, and
$\varepsilon$ is the map induced by the augmentation $\apl B \to \Q$.

Write $\hat \E(E,F)$ for the kernel of the map
$\varepsilon:\E(\apl B \otimes \Lambda V_F, d) \to \E(\Lambda V_F, \bar d)$.
There is an induced quasi-isomorphism $\hat \E(E,F) \xrightarrow{\simeq} \E(\apl E, \apl F)$.
Our work is completed in the lemmas which follow by establishing
that the inclusion maps $i_1$ and $i_2$ in the 
diagram below are both quasi-isomorphisms, and therefore so is the induced map on kernels 
$\E \apl B \to \hat \E(E, F)$. 

$$
\xymatrix@C=40pt{
 \E \apl F & \E\apl E \ar[l]_(.6){i^*} & \E(E,F) \ar[l]_(.35){\delta} \\
 \E(\Lambda V_F) \ar[u]^{\E m_F}_{\simeq} & 
   \E(\apl B \otimes \Lambda V_F) \ar[l]_(.6){\varepsilon_*} \ar[u]^{\E m}_{\simeq} &
   \hat \E(E, F) \ar[l]_(.35){\bar \delta} \ar[u]_{\simeq}  \\ 
 V_F \ar[u]^{i_1}_{\simeq} &
   \E \apl B \oplus V_F \ar[l]_(.6){pr} \ar[u]^{i_2}_{\simeq} & 
   \E \apl B \ar[l]_(.35){\hat \delta} \ar[u]_{\simeq} 
	  \ar@/_/[ul]^{\iota_*} \ar@/_/[uul]|(.46){\phantom{X^X_X}}_(.7){p^*} }
$$
where the differentials on $V_{F}$ and $\Lambda V_{F}$ are the standard  restrictions of that on 
$\apl B \otimes \Lambda V_{F}$.

\end{proof}

An alert reader would note that while we have generally steered clear of the 
traditional approach to rational homotopy through minimal models, we do use the Sullivan model
for a fibration here.  We only need that the cochains of a 
total space of a fibration is a twisted tensor product of those on the
fiber and base, and we consider Sullivan model theory to be the 
most convenient and well-known place to reference this fact.  
We could just as well have used earlier work of Brown \cite{Brow59} giving this result for simplicial
cochains of fibrations, along
with a translation to PL cochains and 
the fact that we can make free models (we did not require minimality).
Thus, our work is still independent from the theory of minimal models.

\begin{lemma} \label{L:BTV=V} 
The inclusion $V_F  \to \E(\Lambda V_F)$ sending $v$ to $\overtie{v}{\bullet}$
is a quasi-isomorphism.
\end{lemma}


\begin{proof}
We proceed in a  similar manner as Theorem 19.1 of \cite{FHT01}.
The homogeneous elements of $\E\Lambda V_F$ are graphs whose vertices
are decorated by nonempty words in $V_F$, the sum of whose lengths we call
the total word length.
We define a map of complexes $h : \E\Lambda V_F \to \E\Lambda V_F$ which,
away from the image of $V_F$
gives a chain homotopy between the identity map and a 
map which increases the weight of a graph while keeping
its total word length fixed if possible, or which is zero if that is not possible.
Since weight is bounded by total word length, iterated applications of this will
eventually produce a null-homology of any cycle not in $V_F$. 

Write a generic long graph $\gamma \in \E^{m+1}\Lambda V$ of
total word length $n$ as
$$\gamma =  \left(
\begin{xy}
 (-10,-3)*+UR{\scriptstyle v_1 v_{2} \cdots}="1",
 (0,3)*+UR{\scriptstyle v_{k_1} v_{k_{1}+1} \cdots}="2",
 (10,-3)*+UR{\cdots}="3",
 (20,3)*+UR{\scriptstyle v_{k_j} \cdots}="4",
 (30,-3)*+UR{\cdots}="5",
 (40,3)*+UR{\scriptstyle v_{k_m} \cdots v_n}="6",
 "1";"2"**\dir{-}?>*\dir{>},
 "2";"3"**\dir{-}?>*\dir{>},
 "3";"4"**\dir{-}?>*\dir{>},
 "4";"5"**\dir{-}?>*\dir{>},
 "5";"6"**\dir{-}?>*\dir{>},
\end{xy}\right).$$
Define $h(\gamma)$ to be zero if the total word length of $\gamma$ is 
equal to $m + 1$ (in which case $\gamma = v_{1} | \cdots | v_{m+1}$).  Otherwise $\gamma$ 
has at least one word of length two decorating some vertex.  In this case  
$h(\gamma)$ is a sum, over the letters in such words, of 
the graphs obtained by 
removing each letter in turn and using it to decorate a new vertex 
attached to the old vertex
while fixing the rest of the graph.  
Explicitly,
$$h(\gamma) \ = \  
\displaystyle \frac{1}{n} \sum_{\substack{j \text{ with } \\ k_{j+1} - k_j > 1}}  
               \sum_{k_j \le p \le k_{j+1}-1} \ (-1)^{\kappa(p)} \ \  
%
 \Biggl(\begin{aligned}\begin{xy}
  (0,-3)*+UR{\cdots}="1",
  (18,3)*+UR{\scriptstyle v_{k_j} \cdots  \widehat{v}_p  \cdots   v_{k_{j+1}-1}}="2",  
  (36,-3)*+UR{\cdots}="3",
  (18,10)*+UR{\scriptstyle v_p}="4",
  "1";"2"**\dir{-}?>*\dir{>},
  "2";"3"**\dir{-}?>*\dir{>},
  "4";"2"**\dir{-}?>*\dir{>}
 \end{xy}\end{aligned}\Biggr) 
%
%
$$ 
where the $(-1)^{\kappa(p)}$ is the Koszul sign coming from moving a degree
one operator to $v_p$'s vertex and moving $v_p$ to the front of its word and 
across an $s^\inv$ and $n$ is the total word length.

More generally if $\gamma$ is any decorated graph, we define $h$ as a
sum over the letters in the
words of length at least two decorating the vertices of $\gamma$, removing each letter in
turn and using it to decorate
a new vertex as above. 
The internal ordering
of the  vertices in the resultant graph is the same as for $\gamma$, with the
new vertex occurring immediately before the vertex its letter was removed from.

That $h$ is well defined
follows easily from the local nature of the anti-symmetry and Arnold relations in the definition
of $\E$.  
Away from $V_F$, $h$ decreases degree by one (due to the unwritten $s^\inv$ in front of $v_p$) and increases graph length by one.

It is straightforward to check that $d_{\E} h + h d_{\E} = id + \text{(graphs of greater length)}$ 
outside of $V_F$.  
The Arnold and arrow reversing relations are required in order to establish 
equality for graphs with internal vertices decorated by singleton words.  For example, 
using notation from \cite{SiWa06} and writing $(-1)^a$ instead of
$(-1)^{|a|}$, we have  
$$\begin{aligned}
d_{\E}\, h \left(\biggraphpp{s^\inv a}{s^\inv b}{s^\inv c}\right) &= d_{\E}\, 0 = 0\\
h\, d_{\E} \left(\biggraphpp{s^\inv a}{s^\inv b}{s^\inv c}\right)
&= h \left( 
	(-1)^{a}
 \begin{xy} 
  (0,1)*+R{\scriptstyle s^\inv ab}="ab",
  (12,-1)*+R{\scriptstyle s^\inv c}="c",
  "ab";"c"**\dir{-}?>*\dir{>}, 
 \end{xy}
+ (-1)^{a+1+b}
 \begin{xy} 
  (0,-1)*+R{\scriptstyle s^\inv a}="a", 
  (12,1)*+R{\scriptstyle s^\inv bc}="bc", 
  "a";"bc"**\dir{-}?>*\dir{>}, 
 \end{xy}                   
 + d\!_A \bigl(\biggraphpp{s^\inv a}{s^\inv b}{s^\inv c}\bigr)
\right) \\
&= (-1)^a \, \frac{1}{3} \left( 
(-1)^{a}
\overtie{
 \begin{xy} 
	(0,4)*+UR{\scriptstyle 1}="a",
  (0,-2)*+UR{\scriptstyle 2}="b",
  (8,-3)*+UR{\scriptstyle 3}="c",
  "a";"b"**\dir{-}?>*\dir{>}, 
  "b";"c"**\dir{-}?>*\dir{>}, 
 \end{xy}
}{s^\inv a\otimes s^\inv b \otimes s^\inv c}
+ (-1)^{ab+b} 
\overtie{
 \begin{xy} 
	(0,4)*+UR{\scriptstyle 1}= "b",
  (0,-2)*+UR{\scriptstyle 2}="a",
  (8,-3)*+UR{\scriptstyle 3}="c",
  "b";"a"**\dir{-}?>*\dir{>}, 
  "a";"c"**\dir{-}?>*\dir{>}, 
 \end{xy}
}{s^\inv b\otimes s^\inv a\otimes s^\inv c}
\right)  \\
& \qquad  +
(-1)^{a+1+b} \, \frac{1}{3}\left( 
(-1)^{a+1+b}
\overtie{
 \begin{xy} 
  (0,-3)*+UR{\scriptstyle 1}="a",
  (8,-2)*+UR{\scriptstyle 3}="c",
	(8,4)*+UR{\scriptstyle 2}="b",
  "a";"c"**\dir{-}?>*\dir{>}, 
  "b";"c"**\dir{-}?>*\dir{>}, 
 \end{xy}
}{s^\inv a\otimes s^\inv b \otimes s^\inv c}
+ (-1)^{a+1+bc+c} 
\overtie{
 \begin{xy} 
  (0,-3)*+UR{\scriptstyle 1}="a",
  (8,-2)*+UR{\scriptstyle 3}="b",
	(8,4)*+UR{\scriptstyle 2}="c",
  "a";"b"**\dir{-}?>*\dir{>}, 
  "c";"b"**\dir{-}?>*\dir{>}, 
 \end{xy}
}{s^\inv a\otimes s^\inv c\otimes s^\inv b}
\right) + 0 \\
&= \frac{2}{3} \biggraphpp{s^\inv a}{s^\inv b}{s^\inv c} - 
\frac{1}{3}\left( 
(-1)^{(a+b)(c+1)} \biggraphpp{s^\inv c}{s^\inv a}{s^\inv b} 
+ (-1)^{(a+1)(b+c)} \biggraphpp{s^\inv b}{s^\inv c}{s^\inv a}\right)
 = \biggraphpp{s^\inv a}{s^\inv b}{s^\inv c}\end{aligned}$$
The general proof proceeds by ``adding 
whiskers'' to the input vertices of these graphs, so that this calculation becomes a local 
part of a larger graph.  The Arnold identity is used an additional time to move whiskers back
to their correct vertex.

In the end, on the basic case of graphs whose total word length is $m+1$ 
$h$ is a chain homotopy between the identity map and the projection onto $V_{F}$ .
\end{proof}

\begin{remark}
Our proof of Lemma~\ref{L:BTV=V} makes essential use of the graph complex representation
of Harrison homology.
In particular, the definition of $h$  above does not preserve long graphs, 
so it is not defined on the bar complex representation of Harrison homology.  
The standard proof of the bar complex analogue of Lemma~\ref{L:BTV=V} uses the homotopy
$h(v_1\otimes\cdots \otimes v_{k_1}|\cdots|\cdots \otimes v_n) = 
 (v_1|v_2\otimes \cdots \otimes v_{k_1}|\cdots|\cdots \otimes v_n)$ if $k_1\neq 1$ and 
$0$ otherwise.  This does not induce a well-defined map either from commutative algebras 
(using symmetric tensors) or to bar expressions modulo shuffles.
Thus for example \cite{FHT01} Proposition 22.8, which is approximately dual to Lemma~\ref{L:BTV=V},
must be proven by non-constructive methods, detouring through the universal enveloping algebra.
\end{remark}

\begin{lemma}\label{L:l2}
The inclusion $\E\apl B \oplus V_F \to \E(\apl B\otimes \Lambda V_F)$
is a quasi-isomorphism.
\end{lemma}

\begin{proof}[Sketch of proof]
The proof is similar to the previous lemma.
Homogeneous elements of $\E(\apl B \otimes \Lambda V_F)$ are graphs whose vertices
are decorated by an element of $\apl B$ and a word in $V_F$.  On any vertex, the decorating 
element from $\apl B$ may be trivial or the $V_F$ word may be empty, but not both at once.

As in the proof of Lemma~\ref{L:BTV=V}, we define a chain  homotopy $h$ which
evaluates on a graph as a sum over replacements of each vertex  
``pulling out'' $V_F$ letters in turn, connecting them to the vertex by a new edge.
As before it is straightforward to calculate that $d_{\E} h + h d_{\E} = id + \text{(graphs of greater length)}$
outside of $\E \apl B \oplus V_F$.
\end{proof}

To our knowledge, the dual of this result does not appear in the literature.
Attempting to directly prove the dual 
statement by methods analogous to the proof of Proposition 22.8 in \cite{FHT01} would be 
difficult, though 
it could be deduced through more abstract methods we require an explicit definition
of chain homotopies for computations, such as in Section~\ref{S:hom}.

\begin{proposition}
Evaluation of Hopf invariants is compatible with the isomorphisms  $\pi_{n}(E, F) \cong \pi_{n}(B)$
and $H^{n}_{\E}(E, F) \cong H^{n}_{\E}(B)$.  That is, $\langle p^{*} \gamma,\, f \rangle_{(E, F)}
= \langle \gamma,\, p_{*} f \rangle_{B}$ and 
$\langle \gamma,\, \lambda f \rangle_{(E, F)} = \langle \bar \gamma,\, f \rangle_{B} .$
\end{proposition}

\begin{proof}
The first equality is immediate since both $\langle p^{*} \gamma,\, f \rangle_{(E, F)}$ and 
$\langle \gamma,\, p_{*} f \rangle_{B}$ 
are equal to $\int_{D^{n}} \tau\left((p \circ f)^*\gamma\right)$.

For the second equality, we have
$$\langle \gamma,\, \lambda f \rangle_{(E, F)} = \int_{D^{n}} \tau\left((\lambda f)^* \gamma\right) 
= \int_{D^{n}} \tau\left((\lambda f)^* p^{*} \bar \gamma\right)
= \int_{D^{n}} \tau\left((p \circ \lambda f)^*\bar \gamma\right) 
= \int_{D^{n}} \tau(f^* \bar \gamma) = 
\langle \bar \gamma,\, f \rangle_{B}.$$
\end{proof}

We can now quickly prove Theorem \ref{T:complete}.

\begin{proof}[Proof of Theorem \ref{T:complete}]
Applying the perfect pairing between the $\E$ and homotopy long exact sequences of a fibration  
to the Postnikov tower of a space, it is enough to know that $\eta^\E$ gives an isomorphism on 
Eilenberg-MacLane spaces. 

Recall that $\apl K(\Q^m, n)$ is quasi-isomorphic to a free 
graded algebra with $m$ generators in degree $n$.  Lemma~\ref{L:BTV=V} gives a 
null homotopy of all elements of $\E(\Lambda s^n \Q^m)$ except for these generators.
The Hopf invariants associated to these generators, given by simply pulling back cohomology
and evaluating,  are linearly dual to $\pi_{n}(K(\Q^m, n))$ by the Hurewicz theorem.   
 \end{proof}

We need the existence but not the uniqueness of the Postnikov tower.  
On the whole, the only theorems from topology which
we use to define our functionals and prove they are 
complete are this existence and 
the fact that a fibration can be modeled by a twisted tensor
product of cochain models.  The rest of the basic input has been algebraic, namely
our development of the Lie cooperadic bar construction.

\section{Applications and open questions}

We give applications and present some questions now opened up to investigation.  
We have listed one such already,
namely the generalized Hopf invariant one question.  Our work
in the previous section gives rise to a new approach to this problem.
Namely one could try to understand how 
the classical bar construction over the integers, modulo Harrison shuffles, fails to yield exact sequence of a fibration.  
Such an approach could  give rise to a new, very different
proof of Adams' result.   Answering the question for general $X$ would yield significant insight into
the relationship between homotopy theory in characteristic zero and characteristic $p$.

\subsection{Wedges of spheres}
Revisiting Example~\ref{E:wedge1},  recall that a model 
for cochains on a wedge of spheres $X$  is just the graded vector space $W$ with one
generator $w_{i}$ in the appropriate dimension for each sphere,  with trivial products
and differential.   \refT{complete} now implies
that the linear dual of $\pi_{*}(X)$ is just the cofree Lie coalgebra on $W$.
We take  advantage of \refC{cobracket} and the perfect pairing between free Lie
algebras and cofree Lie coalgebras to deduce the Hilton-Milnor Theorem that
 the homotopy groups of $X$ are just the free Lie algebra on $W$.

We not only know what the homotopy Lie algebra is but we 
see explicitly how to determine
whether two maps into wedges of spheres are homotopic.  The vector space
$\E(W)$ is spanned by cocycles of the form 
$\gamma_{\sigma} = 
\begin{xy}
 (-10,-2)*+UR{\scriptstyle w_{\sigma(1)}}="1",
 (0,2)*+UR{\scriptstyle w_{\sigma(2)}}="2",
 (9,-2)*+UR{\cdots}="3",
 (18,2)*+UR{\scriptstyle w_{\sigma(k)}}="4",
 "1";"2"**\dir{-}?>*\dir{>},
 "2";"3"**\dir{-}?>*\dir{>},
 "3";"4"**\dir{-}?>*\dir{>}
\end{xy}
$ as $\sigma$
varies over the symmetric group.  Following Section~\ref{S:mfld} 
represent $w_{i}$ by the  
Thom class  associated to a point $p_{i}$ (away from the basepoint) of each 
wedge factor $S^{d_{i}}$.   Given some map $f: S^{n} \to X$  we let 
$W_{i} = f^{-1}(p_{i}) \subset \R^{d_{i}}$.  
The argument in Section~\ref{S:mfld} applies, and  the Hopf invariant
of $\gamma_{\sigma}$ is a linking number of the $W_{i}$,
counting the number of times (with signs) that a point in
$W_{\sigma(1)}$ lies ``directly above'' one in $W_{\sigma(2)}$ which in turn lies above
one in $W_{\sigma(3)}$, etc.  
The rational homotopy class of $f$ is determined by such counts, and can be named as an element
of the free Lie algebra using the configuration pairing.

\subsection{Homogeneous spaces}\label{S:hom}

Consider the standard Puppe sequence 
(for fibrations, as in \cite{Nomu60})
$$H \overset{i}{\to} G \overset{p}{\to} G/H \overset{f}{\to} BH \overset{Bi}{\to} BG.$$
Here $H$ and $G$ are $H$-spaces, $i$ is a map of $H$-spaces,
and $G/H$ is defined as the homotopy fiber
of $Bi$, which models the corresponding homogeneous space when $i$
is an inclusion of Lie groups.
Rationally an $H$-space is a generalized 
Eilenberg-MacLane space, which is formal with
free rational cohomology algebra.  By \refL{BTV=V},
we have for example that  $\E(G) \simeq V_{G}$, where $V_{G}$ is the vector space of  indecomposables, 
represented
as weight one cocycles.  The associated Hopf invariants are simply evaluation of these cohomology 
classes, as expected in the Eilenberg-MacLane setting.

If we also identify $\E(H) \simeq V_{H}$ then ${\rm Hom}(\pi_{*}(G/H), \Q)$, which is isomorphic 
to $H_{\E}^{*}(G/H)$ if $G/H$ is simply connected, is isomorphic to the direct sum
of the kernel and suspended cokernel of $i^{*} : V_{G} \to V_{H}$.  (Here to suspend and 
obtain classes in $H^{*}(BG)$ and $H^{*}(BH)$ we can 
for example use the transgression in the Leray-Serre spectral sequence.)
Hopf invariants associated to the 
cokernel of $i^{*}$ are simple to interpret.  A cohomology class in $V_{H}$ suspends to a cohomology class
of $BH$ which  can be pulled back
to $G/H$ and evaluated on homotopy.   This will be trivial if the class was already pulled back from $BG$.

On the other hand consider $x \in V_{G}$, which we identify with a representative weight
one cocycle in $\E(G)$, with $i^{*}(x) = 0$.  In the model $\apl G/H \otimes \Lambda V_{H}$ for 
the cochains of $G$, represent $x$ by $a_{0}\otimes 1 + \sum  a_{i} \otimes f_{i}$ with the $a_{i}$ closed
and $f_{i} \neq 1$.  
Then in $\E (\apl G/H \otimes \Lambda V_{H})$
we have 
$${x} - d_{\E}(\sum a_{i} | f_{i})  = a_{0} + \sum  a_{i} | b_{i},$$
 where $b_{i} = d f_{i}$ is in $\apl G/H$ since $\Lambda V_{H}$ has trivial differential.
 Thus $x$ is homologous to the pull-back through $p$ of $a_{0} + \sum a_{i} | b_{i} 
 \in \E(G/H)$.  (A useful exercise here is to consider the Hopf maps.)

The Hopf invariants for $G/H$ - which are complete when $G/H$ is simply connected - 
are therefore given by either evaluation of
cohomology or by classical linking invariants of the pullbacks of the Poincar\'e duals
of the $a_{i}$ and $b_{i}$ as in Section~\ref{S:mfld}.  
When $H$ and $G$ are Lie groups, one can also use de Rham  theoretic models of all of these spaces
defined through the Lie algebras of $H$ and $G$ (the antecedents of Quillen functors) for
explicit calculations.

 While
weight one cocycles account for the Hopf invariants of a generalized Eilenberg-MacLane
space, weight two cocycles were required here.  In general the highest
weight which appears in  $H^{*}_{\E}(X)$ should reflect the number of stages in the smallest
tower atop which the rationalization of
$X$ sits whose fibers are generalized Eilenberg-MacLane spaces.

 \subsection{Configuration spaces}
Next consider the ordered configuration space $\rm{Conf}_{n}(\R^{d})$, with
$n=3$.  Its cohomology algebra is generated
by classes $a_{12}$, $a_{13}$ and $a_{23}$ where for example $a_{12}$
is pulled back from the map which sends $(x_{1}, x_{2}, x_{3})$  to 
$\frac{x_{2} - x_{1}}{||x_{2} - x_{1}||} \in S^{d-1}$.  These cohomology
classes satisfy the Arnold identity $a_{12}a_{23} + a_{23}a_{31} + a_{31}a_{12} = 0$
(the same identity at the heart of our approach to Lie coalgebras, an overlap
which we explain below).   Its homotopy Lie algebra is generated
by classes whose images in homology are dual, which we denote $b_{ij}$,
for which ``$x_{i}$ and $x_{j}$ orbit each other.''   Their Lie brackets
satisfy the identity that $[b_{12}, b_{23}] = [b_{23}, b_{31}] = [b_{31}, b_{12}]$.  

The best way to construct Hopf invariants is not with the standard $a_{ij}$
since their products are non-trivial at the cochain level.  The simplest cocycle in 
in $\E( \rm{Conf}_{3}(\R^{d}))$ using these cochains would be
$$\linep{a_{12}}{a_{23}} + \linep{a_{23}}{a_{31}} + \linep{a_{31}}{a_{12}} 
+\ \overset{\theta}{\bullet},$$
where $\theta$ is a cochain cobounding the Arnold identity.
Consider instead the submanifold
of points $(x_{1}, x_{2}, x_{3})$ which are collinear.  This submanifold has three 
components, which we label by which $x_{i}$ is 
``in the middle.''  These are proper submanifolds, so we consider their Thom
classes, which we denote $Co_{i}$.  By intersecting these with the cycles
representing the homology generators $b_{ij}$ we see that for example
$Co_{1}$ is  cohomologous to $a_{31} + a_{12}$.  Because they
are disjoint, we have that any graph 
$\begin{xy}       
  (0,-2)*{\scriptstyle Co_i}="a",
  (6,2)*{\scriptstyle Co_j}="b", 
  "a";"b"**\dir{-}?>*\dir{>},    
\end{xy}$    
is a cocycle in $\E(\rm{Conf}_{n}(\R^{d}))$.  We may use
\refC{cobracket} to see for example that 
$$\begin{aligned}
\left\langle 
 \begin{xy}       
  (0,-2)*{\scriptstyle Co_1}="a",
  (6,2)*{\scriptstyle Co_2}="b", 
  "a";"b"**\dir{-}?>*\dir{>},    
 \end{xy},\    
     [b_{12}, b_{23}] \right\rangle_{\eta}
  &= (a_{31} + a_{12})(b_{12})\cdot (a_{12} + a_{23})(b_{23}) 
     +  (a_{12} + a_{23})(b_{12})  \cdot ( a_{31} + a_{12})(b_{23}) \\
  &= 1 \cdot 1 + 1 \cdot 0 = 1.
\end{aligned}$$
Thus, to understand an element $f$ of $\pi_{2d-3}( \rm{Conf}_{3}(\R^{d}))$
it suffices to understand the linking behavior of $f^{-1}(Co_{i})$.
This calculation is reflected in the results of \cite{BCSS05}, which
used relative Hopf invariants of an evaluation (or Gauss) map
for knots to give a new interpretation of the simplest finite-type knot invariant.
Understanding the framework of how such homotopy invariants can be defined and evaluated on
more complicated Whitehead products is a primary motivation for our development
of Hopf invariants.

\subsection{The graph complex}

Though it played only an incidental role in our current development, all of our work in
the bar complex and some of our work in the Lie coalgebra complex can be shifted onto 
the graph complex $\mathcal{G}(X)$.  
In fact, from the point of view of $\E(X)$, the graph complex is 
a much more natural home for the development of Hopf invariants than the bar complex.

Constructions and proofs proceed in the same manner as those for the bar complex,
with arguments establishing basic properties applying verbatim.
Equipping the graph complex with the anti-cocommutative graph-cutting cobracket,
the proof of 
\refT{cobracket} applies to give compatibility with Whitehead products.  Note that
\refT{cobracket} requires only a component bigrading and internal differential argument.
In the graph complex, components are given by maximal size subgraphs.

Furthermore, since the quotient map $\mathcal{G}(S^n) \to \E(S^n)$ is the identity 
on weight one, the proof of \refT{splitcompat} applies to show that 
$\eta^{\mathcal{G}} = \eta^{\E} \circ p$.  
However, if we wished to use $\eta^{\mathcal{G}}$ to show that $\eta^{\E}$ is well defined
we would either need an alternate, direct, proof that Hopf invariants vanish on 
arrow reversing and Arnold expressions, or else we would need a splitting of the 
quotient map $\mathcal{G}(A) \to \E(A)$, whose existence is open at the moment.
In modern terminology, the splitting of the quotient
$B(A) \to \E(A)$ is given by
the first Eulerian idempotent $e^{(1)}_*$ (\cite{Loda98} \S 4.5). 
Together with the other Eulerian idempotents,
this splits the bar complex into the Hodge decomposition.  It is tempting to seek
a similar decomposition of the graph complex, which might give an 
alternate way of understanding the Eulerian idempotents.  

\subsection{The non-simply-connected setting}

In future work we plan to extend these results beyond the simply connected case.
One main reason we consider only simply connected spaces here is that Lie coalgebras
are not well-understood outside of the (simply) connected setting, unless 
there is a nilpotence condition in effect.  In \cite{Walt08}, the first author is putting
coalgebras on sounder footing.  We expect results even beyond the nilpotent setting, 
as we have seen in preliminary calculations that Hopf invariants, while more delicate
to define, do distinguish the rational homotopy
groups of $S^{1} \vee S^{2}$.

It would be interesting to compare our Hopf invariants with homotopy invariants
coming from $H^{0}$ of the generalized Eilenberg-Moore spectral sequence for
${\rm Map}(X, Y)$.  Such a comparison might be useful in extending our techniques
to  give invariants of homotopy classes of maps
from $X$ to $Y$ more generally.

\section{Comparison with other approaches}\label{S:compare}

Our theorems give a resolution of the ``homotopy period'' problem
of explicitly representing homotopy functionals, a classical question considered many times over the past seventy-five years.  A formula for homotopy periods using the same ingredients, namely pulling back forms to the sphere, taking $\dinv$ and wedge products, and then integrating,  was featured on the first
page and then in Section~11
of Sullivan's seminal paper \cite{Sull77}.  
We have not been able to compare our formula with Sullivan's, since that
formula comes from a minimal 
model for path spaces which uses a chain homotopy between an isomorphism and the zero map
which is also a derivation.
In some examples, we cannot find such a chain homotopy, and if one relaxes
the condition of being a derivation then the inductive formula is not clear.

In the same vein,  some of the lower weight cases of our constructions were treated
in work of Haefliger \cite{Haef78} and Novikov \cite{Novi88}.  Haefliger in particular gave
formulae for Hopf forms which are special cases of Examples~\ref{E:arbwt1}~and~\ref{E:wt2}.
He also gives formulae for evaluating these forms on Whitehead products,
which of course follow from our \refT{cobracket}.  Haefliger's comment was that ``It is
clear that one could continue this way for higher Whitehead products, if one is not afraid of
complicated formulas.''   The bracket-cobracket formalism and the configuration pairing
make these formulae simple, at least conceptually.  Haefliger uses Hopf invariants associated to cocycles 
in $\E(X)$ of the form $\omega_{1} | \cdots | \omega_{n}$ where $\omega_{i} \omega_{i-1} = 0$
to show that such $X$ have summands of the corresponding
free Lie algebra in their homotopy groups, which is also immediate from our approach.

As mentioned above, Hain's thesis \cite{Hain84} solves the rational homotopy period 
using Chen integrals.  
Here one appeals to the Milnor-Moore theorem that
identifies rational homotopy within loopspace homology \cite{MiMo65}, so that dually loopspace
cohomology must give all homotopy functionals.  The main work is to find an irredundant
collection.  The geometric heart of Chen integrals are the evaluation maps
$\Delta^{n} \times \Omega X \to X^{n}$ which evaluate a loop at $n$ points.  Thus when writing down these functionals explicitly, the integrals are
over $S^{n} \times \Delta^{k}$ not $S^{n}$ itself.
Our approach is thus distinct and more intrinsic.  Indeed,
understanding homotopy solely through loopspace homology is understanding a Lie algebra
through its universal enveloping algebra.  

As mentioned in \ref{S:mfld}, our Hopf invariants also connect with those 
of Boardman and Steer \cite{BoSt66, BoSt67}.  
Their Hopf invariants are maps $\lambda_{n} : [\Sigma Y,\, \Sigma X] \to [\Sigma^{n} Y,\, \Sigma
X^{\wedge n}]$.  In some cases, the latter group can be computed, yielding homotopy
functionals.  A key intermediate they take is defining
a map $\mu_{n}: [\Sigma Y,\, \bigvee_{i=1}^{n} X_{i}]
\to [\Sigma^{n} Y,\, \bigwedge_{i=1}^{n} X_{i}]$.  Comparing the proof of Theorem~6.8 of 
\cite{BoSt67} with our \ref{S:mfld} we see that their functionals agree with 
ours in this setting.  While their technique is more general in the direction of analyzing maps out
of any suspension as opposed to only spheres, 
our approach allows analysis of homotopy groups of  spaces which are not  suspensions.  
Indeed, the behaviour
$\E(\Sigma X)$ is fairly trivial, since finite sets of cochains may be assumed to have
disjoint support, so the bar complex collapses. 

Finally, we follow up on Example~\ref{E:wedge1}  where we showed that the value
of one of our Hopf invariants on a Whitehead product is given by the configuration
pairing between free Lie algebras and coalgebras.  This pairing 
appears  in the homology and cohomology of ordered configuration
spaces.  Recall from Section~\ref{S:mfld}  that if $\{ W_{i} \}$ is a
disjoint collection of proper submanifolds of $X$ and $f$ is a map $S^{d} \to X$, 
then our Hopf invariants are encoding the map on homology induced by inclusion of 
$\prod f^{-1}(W_{i})$ in ${\rm Conf}_{k}(\R^{d})$.  
As Fred Cohen often says, ``this cannot be a coincidence.'' In this case the
connection is readily explained through Cohen's work on the homology of iterated 
loopspaces. 

Consider $Y = \bigvee_{i=1}^{k} S^{d_{i} + 2}$ and let $n = d_{1} + \cdots + d_{k} + k -1$,
the dimension in which one can see a Whitehead product $P$ of the inclusion maps $\iota_{i}$ with
each map appearing once.  Then $\pi_{n}(Y) \cong \pi_{n-2}(\Omega^{2} \Sigma^{2} 
(\bigvee S^{d_{i}}) ) $ which maps to $H_{n-2} (\Omega^{2} \Sigma^{2} 
(\bigvee S^{d_{i}}) )$ under the Hurewicz homomorphism.    Loopspace
theory identifies this homology with that of 
$\bigcup_{m}{\rm Conf}_{m}(\R^{2}) 
\wedge_{\Sigma_{m}} (S^{d_{1}} \vee \cdots \vee S^{d_{k}})^{\wedge m}$, which  in turn has
a summand coming from the homology of 
${\rm Conf}_{k}(\R^{2}) \times_{\Sigma_{k}} \left( \bigcup_{\sigma \in \Sigma_{k}} 
S^{d_{\sigma(1)}} \wedge \cdots \wedge S^{d_{\sigma(k)}} \right)$.  But since the $\Sigma_{k}$ action
is free on the union of products of spheres, we get that this is simply 
$$H_{n-2}({\rm Conf}_{k}(\R^{2}) \wedge S^{\sum d_{i}}) \cong H_{k-1}({\rm Conf}_{k}(\R^{2})) \cong
\lie(n).$$  (See \cite{Sinh06} for an expository treatment of this last isomorphism.)
In Remark~1.2 of \cite{CLM76},  established at the end of Section~13, Cohen gives
a diagram establishing the compatibility of Whitehead products and Browder brackets,
through the Hurewicz homomorphism.  This result
implies that  a Whitehead product $P$ maps to this $\lie(k)$ summand in the canonical way,
going to the homology class with the same name (up to sign).

In Example~\ref{E:wedge1}, we showed that the Hopf invariant $\eta_{\gamma}$ evaluated
on $P$ according to the configuration pairing between free Lie algebras and coalgebras,
modeled in the operad/cooperad pair $\lie$ and $\eil$.  Thus a cohomology
class $g \in \eil(k) \cong H^{k-1}(\rm{Conf}_{k}(\R^{2}))$ represented by a graph
(see \cite{Sinh06})  gets mapped 
under the linear dual of the Hurewicz
homomorphism to essentially the same graph (with each vertex label replaced by a volume form
on the corresponding sphere) in $\E(Y)$.  

In summary, for wedges of spheres
and other double-suspensions our Hopf invariants coincide with applying the Hurewicz
homomorphism and then evaluating on the cohomology of configuration spaces, which explains
the combinatorial similarity between our Hopf invariants and that cohomology.
At one point we outlined a proof of \refC{cobracket} using this kind of 
loopspace machinery, but we
later found the more elementary approach given here.

 \bibliographystyle{amsplain}
 \bibliography{references}

\end{document}